\documentclass[british]{scrartcl}
\usepackage[T1]{fontenc}
\usepackage[latin9]{inputenc}
\usepackage{babel}
\usepackage{verbatim}
\usepackage{prettyref}
\usepackage{mathtools}
\usepackage{amsmath}
\usepackage{amsthm}
\usepackage{amssymb}
\usepackage{esint}
\usepackage[unicode=true,pdfusetitle,
 bookmarks=true,bookmarksnumbered=false,bookmarksopen=false,
 breaklinks=false,pdfborder={0 0 1},backref=false,colorlinks=false]
 {hyperref}
\usepackage{breakurl}

\makeatletter
 \theoremstyle{definition}
 \newtheorem*{defn*}{\protect\definitionname}
\theoremstyle{plain}
\newtheorem{thm}{\protect\theoremname}[section]
  \theoremstyle{remark}
  \newtheorem{rem}[thm]{\protect\remarkname}
  \theoremstyle{plain}
  \newtheorem{prop}[thm]{\protect\propositionname}
  \theoremstyle{plain}
  \newtheorem{cor}[thm]{\protect\corollaryname}
  \theoremstyle{definition}
  \newtheorem{example}[thm]{\protect\examplename}
  \theoremstyle{plain}
  \newtheorem{lem}[thm]{\protect\lemmaname}
  \theoremstyle{definition}
  \newtheorem{defn}[thm]{\protect\definitionname}

\@ifundefined{date}{}{\date{}}
\usepackage{prettyref}
\usepackage{enumerate}

\allowdisplaybreaks

\newcommand{\spt}{\operatorname{spt}}
\newcommand{\sym}{\operatorname{sym}}
\newcommand{\essran}{\operatorname{essran}}
\newcommand*{\e}{\mathrm{e}}
\renewcommand*{\i}{\mathrm{i}}
\newcommand{\m}{\operatorname{m}}
\renewcommand{\tilde}{\widetilde}
\renewcommand{\hat}{\widehat}
\renewcommand{\Re}{\operatorname{Re}}
\renewcommand{\Im}{\operatorname{Im}}
\DeclareMathOperator*{\wlim}{w-lim}

\newrefformat{prop}{Proposition \ref{#1}}
\newrefformat{lem}{Lemma \ref{#1}}
\newrefformat{thm}{Theorem \ref{#1}}
\newrefformat{cor}{Corollary \ref{#1}}
\newrefformat{rem}{Remark \ref{#1}}
\newrefformat{eq}{(\ref{#1})}
\newrefformat{item}{(\ref{#1})}
\newrefformat{sec}{Section \ref{#1}}

\newenvironment{keywords}{ \noindent\footnotesize\textbf{Keywords and phrases:}}{}

\newenvironment{class}{\noindent\footnotesize\textbf{Mathematics subject classification 2010:}}{}

\makeatother

  \providecommand{\corollaryname}{Corollary}
  \providecommand{\definitionname}{Definition}
  \providecommand{\examplename}{Example}
  \providecommand{\lemmaname}{Lemma}
  \providecommand{\propositionname}{Proposition}
  \providecommand{\remarkname}{Remark}
\providecommand{\theoremname}{Theorem}

\begin{document}

\title{On Differential-Algebraic Equations in Infinite Dimensions}

\author{Sascha Trostorff \& Marcus Waurick}

\maketitle
\textbf{Abstract. }We consider a class of differential-algebraic equations
(DAEs) with index zero in an infinite dimensional Hilbert space. We
define a space of consistent initial values, which lead to classical
continuously differential solutions for the associated DAE. Moreover,
we show that for arbitrary initial values we obtain mild solutions
for the associated problem. We discuss the asymptotic behaviour of
solutions for both problems. In particular, we provide a characterisation
for exponential stability and exponential dichotomies in terms of
the spectrum of the associated operator pencil. 

\medskip{}


\begin{class} 34A09, 34A12, 34D20, 34D09, 46N20, 47N20 \end{class}

\section{Introduction}

Let $H$ be a Hilbert space and let $M_{0},M_{1}$ be bounded linear
operators in $H$. In this article we are concerned with the implicit
initial value problem of finding $u\colon\mathbb{R}_{\geq0}\to H$
such that
\begin{equation}
\tag{{\ensuremath{\ast}}}\begin{cases}
M_{0}u'(t)+M_{1}u(t)=0, & t>0,\\
u(0+)=u_{0}
\end{cases}\label{eq:IVP0}
\end{equation}
for some given initial datum $u_{0}\in H$, where \prettyref{eq:IVP0}
has to be interpreted in a particular sense. This problem is commonly
known as a ``differential-algebraic equation'' and has attracted
a lot of interest in recent years. Most prominently, for finite-dimensional
$H$, differential-algebraic equations have applications to control
and electrical circuit theory, see e.g.~\cite{Dai1989}. For so-called
regular systems and a finite-dimensional state space $H$, the discussion
of \prettyref{eq:IVP0} can be simplified by applying a generalised
Jordan normal form, see e.g.~\cite{Berger2012,Dai1989,Kunkel2006}.
A similar strategy can be used for particular infinite-dimensional
cases, see \cite{Reis2006}. We are unaware of any other treatment
of differential-algebraic equations in infinite spatial dimensions.
Thus, for the convenience of the reader, we present this work as much as
self-contained as possible. 

Our main motivation for discussing the infinite-dimensional case comes
from the viewpoint of so-called evolutionary equations, a certain
class of partial differential-algebraic equations introduced in \cite{PicPhy}.
The focus of \cite{PicPhy} was to derive a particular Hilbert space
setting such that a large class of linear equations in mathematical
physics can be dealt with in a unified framework. A particular case
is the equation
\begin{equation}
\left(M_{0}u\right)'+M_{1}u+Au=f\label{eq:PDAE}
\end{equation}
for some (unbounded, skew-selfadjoint) operator $A$ in $H$. In order
to cover a large class of problems and to allow for $M_{0}=0$, as
well, the basic setting in \cite{PicPhy} uses homogeneous initial
conditions. This, on the other hand, lead to the question of what
possible initial values can be assumed for \prettyref{eq:PDAE} in
order to have a solution $u$. 

In the habilitation thesis \cite{Trostorff2017a} this question has
been addressed. In this article we shall carry out a more detailed
analysis for the case of when $A=0$. For this reason, we shall furthermore
focus on the case of regular systems with index $0$. Moreover, we
will assume that $R(M_{0})\subseteq H$ is closed.

Under the assumptions mentioned, we will derive a space of admissible
initial values $u_{0}$ such that \prettyref{eq:IVP0} admits a continuous
differentiable (``strong'') solution. Moreover, we will discuss
the regularity of solutions for \prettyref{eq:IVP0}, if $u_{0}$
is not admissible. It turns out that the correct notion of a solution
can be phrased as a ``mild'' solution, that is, $u$ satisfies \prettyref{eq:IVP0}
only in an once integrated sense. Having obtained well-posedness of
\prettyref{eq:IVP0}, we will address the asymptotic behaviour of
both strong and mild solutions. We will formulate a spectral characterisation
of both mild and strong exponential dichotomy as well as mild and
strong exponential stability and will recover the well-known finite-dimensional
stability results, see e.g.~\cite{Griepentrog1986,Dai1989,Stykel2002}.
We emphasise that in our approach we do not use the generalised Jordan
normal form or Weierstrass normal form, simply because these strategies
are not applicable under the assumptions stated. 

A more thorough analysis of the case of non zero index has been initiated
in \cite{Trostorff2017} and will be addressed in future work.

We shall comment on the organisation of this text. In the next section,
we gather some material on evolutionary equations and establish the
time-derivative as a suitable continuously invertible operator in
a weighted vector-valued $L_{2}$-space. In \prettyref{sec:Regular-linear-operator}
we introduce our central object of study and define regular linear
operator pencils (of index $0$). We also provide a characterisation
of regularity given the closedness of the range of $M_{0}$. The characterisation
of admissible initial values $u_{0}$ such that \prettyref{eq:IVP0}
has a continuous differentiable solution will be given in \prettyref{sec:firstIV}.
Mild solutions, that is, where \prettyref{eq:IVP0} holds in an integrated
sense, only, will be considered in \prettyref{sec:secondIV}. It turns
out that the problem of finding continuously differentiable solutions
to \prettyref{eq:IVP0} can be written as an abstract Cauchy problem
leading to a norm-continuous semigroup as a fundamental solution.
The spectrum of the generator of this semigroup is shown to coincide
with the spectrum of the operator pencil in \prettyref{sec:spectra}.
In \prettyref{sec:Asymptotics} we address the asymptotics (exponential
dichotomy) of both the derived Cauchy problem as well as for mild
solutions. Note that the difficult part is to \emph{define} an appropriate
notion for the case of mild solutions as they are a priori only locally
integrable. 

\section{Preliminaries}

Following \cite{PicPhy,KPSTW14_OD,Picard1989}, we introduce the temporal
derivative as a normal operator in an exponentially weighted $L_{2}$-space.
One can view the exponential weight as an $L_{2}$-variant of the
well-known Morgenstern norm that is used for proving the Picard\textendash Lindel\"of
theorem for arbitrary large Lipschitz constant, see \cite{Morgenstern1952}
or \cite[Section 4]{PTW14_OD} for a more recent reference.
\begin{defn*}
Let $\rho\in\mathbb{R}.$ We define the Hilbert space $L_{2,\rho}(\mathbb{R};H)$
of (equivalence classes of) $H$-valued functions as follows 
\[
L_{2,\rho}(\mathbb{R};H)\coloneqq\left\{ f:\mathbb{R}\to H\,;\,f\mbox{ measurable,\,}\intop_{\mathbb{R}}|f(t)|_{H}^{2}\exp(-2\rho t)\,\mathrm{d}t<\infty\right\} ,
\]
equipped with the usual inner product 
\[
\langle f,g\rangle_{\rho}\coloneqq\intop_{\mathbb{R}}\langle f(t),g(t)\rangle_{H}\exp(-2\rho t)\,\mathrm{d}t.
\]
Moreover, we define the operator $\partial_{0,\rho}:H_{\rho}^{1}(\mathbb{R};H)\subseteq L_{2,\rho}(\mathbb{R};H)\to L_{2,\rho}(\mathbb{R};H)$
as the closure of 
\begin{align*}
C_{c}^{\infty}(\mathbb{R};H)\subseteq L_{2,\rho}(\mathbb{R};H) & \to L_{2,\rho}(\mathbb{R};H),\\
\phi & \mapsto\phi'
\end{align*}
where $C_{c}^{\infty}(\mathbb{R};H)$ denotes the space of arbitrarily
differentiable functions from $\mathbb{R}$ to $H$ with compact support.
\end{defn*}
Note that every continuous function $f\colon\mathbb{R}\to H$ with
support bounded below that satisfies the exponential growth condition
$|f(t)|_{H}\leq Me^{\omega t}$ for some $\omega\in\mathbb{R},$ $M\geq0$
and all $t\in\mathbb{R}$ belongs to $L_{2,\rho}(\mathbb{R};H)$ for
all $\rho>\omega$. We shall need this observation later on. Some
more remarks are in order. 
\begin{rem}
\begin{enumerate}[(a)]

\item For $\rho=0$ the space $L_{2,0}(\mathbb{R};H)$ is nothing
but the standard $L_{2}$-space of Bochner-measurable functions with
values in $H$. Moreover, $\partial_{0,0}$ is the usual weak derivative
on $L_{2}.$ 

\item The domain $H_{\rho}^{1}(\mathbb{R};H)$ is an exponentially
weighted variant of the Sobolev space $H^{1}(\mathbb{R};H)$ and becomes
a Hilbert space with respect to the graph inner product of $\partial_{0,\rho}$,
that is, 
\[
\langle f,g\rangle_{\rho,1}\coloneqq\langle f,g\rangle_{\rho}+\langle\partial_{0,\rho}f,\partial_{0,\rho}g\rangle_{\rho}.
\]

\end{enumerate}
\end{rem}

It turns out that the operator $\partial_{0,\rho}$ is a normal operator,
whose spectral representation is given in terms of the so-called Fourier\textendash Laplace
transform, which is defined as follows.
\begin{prop}[{\cite[Corollary 2.5]{KPSTW14_OD}}]
\label{prop:FLtransform} Let $\rho\in\mathbb{R}.$ Then the operator
\begin{align*}
\mathcal{L}_{\rho}|_{C_{c}^{\infty}(\mathbb{R};H)}:C_{c}^{\infty}(\mathbb{R};H)\subseteq L_{2,\rho}(\mathbb{R};H) & \to L_{2,0}(\mathbb{R};H)\\
\phi & \mapsto\left(t\mapsto\frac{1}{\sqrt{2\pi}}\intop_{\mathbb{R}}\exp(-(\i t+\rho)s)\phi(s)\,\mathrm{d}s\right)
\end{align*}
has a unitary extension to $L_{2,\rho}(\mathbb{R};H)$, which will
be denoted by $\mathcal{L}_{\rho}.$ Moreover,
\[
\partial_{0,\rho}=\mathcal{L}_{\rho}^{\ast}(\i\m+\rho)\mathcal{L}_{\rho},
\]
where $\m:D(\m)\subseteq L_{2,0}(\mathbb{R};H)\to L_{2,0}(\mathbb{R};H)$
is given by 
\[
\m f\coloneqq(t\mapsto tf(t))
\]
for $f\in D(\m)\coloneqq\left\{ g\in L_{2,0}(\mathbb{R};H)\,;\,(t\mapsto tg(t))\in L_{2,0}(\mathbb{R};H)\right\} .$ 
\end{prop}

\begin{rem}
In the case $\rho=0,$ $\mathcal{L}_{0}$ coincides with the usual
Fourier transform on $L_{2},$ which is unitary by Plancherel's Theorem.
For general $\rho$ we obtain $\mathcal{L}_{\rho}=\mathcal{L}_{0}\exp(-\rho\m),$
where the operator $\exp(-\rho\m)$ given by $\left(\exp(-\rho\m)f\right)(t)\coloneqq\exp(-\rho t)f(t)$
is obviously unitary from $L_{2,\rho}$ to $L_{2,0}.$ Hence, the
unitarity of $\mathcal{L}_{\rho}$ follows.
\end{rem}

Using the spectral representation for $\partial_{0,\rho}$, we can
easily show the following properties.
\begin{cor}
\label{cor:spadj}Let $\rho\in\mathbb{R}.$

\begin{enumerate}[(a)]

\item The spectrum of $\partial_{0,\rho}$ is given by $\sigma(\partial_{0,\rho})=\left\{ z\in\mathbb{C}\,;\,\Re z=\rho\right\} .$ 

\item The adjoint of $\partial_{0,\rho}$ is given by $\partial_{0,\rho}^{\ast}=-\partial_{0,\rho}+2\rho.$

\item If $\rho\ne0,$ then $\partial_{0,\rho}$ is continuously invertible
with
\[
\left(\partial_{0,\rho}^{-1}f\right)(t)=\begin{cases}
\intop_{-\infty}^{t}f(s)\,\mathrm{d}s & \mbox{ if }\rho>0,\\
-\intop_{t}^{\infty}f(s)\,\mathrm{d}s & \mbox{ if }\rho<0,
\end{cases}\quad(t\in\mathbb{R},f\in L_{2,\rho}(\mathbb{R};H)).
\]

\end{enumerate}
\end{cor}

We conclude the section with the following variant of the Sobolev-embedding
theorem.
\begin{prop}[{\cite[Lemma 5.2]{KPSTW14_OD}}]
\label{prop:Sobolev}Let $\rho\in\mathbb{R}$ and define 
\[
C_{\rho,0}(\mathbb{R};H)\coloneqq\left\{ f:\mathbb{R}\to H\,;\,f\mbox{ continuous, }\,\lim_{t\to\pm\infty}f(t)\exp(-\rho t)=0\right\} ,
\]
equipped with the norm 
\[
|f|_{\rho,\infty}\coloneqq\sup\left\{ |f(t)|_{H}\exp(-\rho t);t\in\mathbb{R}\right\} 
\]
Then $H_{\rho}^{1}(\mathbb{R};H)\hookrightarrow C_{\rho,0}(\mathbb{R};H)$. 
\end{prop}

In the next section, we recall the notion of regularity for differential-algebraic
equations, see e.g.~\cite{Berger2012,Markus1988}. Note that we shall
deviate slightly from the usual notion of regularity in as much as
we restrict our consideration to index 0, only. For differential-algebraic
equations with nontrivial index, we refer to the study initiated in
\cite{Trostorff2017}.

\section{Regular linear operator pencils\label{sec:Regular-linear-operator}}

Throughout this section, let $M_{0},M_{1}\in L(H).$ 
\begin{defn*}
We consider the function 
\begin{align*}
\mathcal{M}:\mathbb{C} & \to L(H)\\
z & \mapsto zM_{0}+M_{1}
\end{align*}
and call $\mathcal{M}$ the \emph{linear operator pencil associated
with }$\left(M_{0},M_{1}\right)$. We define the \emph{spectrum} of
$\mathcal{M}$ by 
\[
\sigma(\mathcal{M})\coloneqq\left\{ z\in\mathbb{C}\,;\,0\in\sigma(zM_{0}+M_{1})\right\} 
\]
and the \emph{resolvent set} of $\mathcal{M}$ by
\[
\varrho(\mathcal{M})\coloneqq\mathbb{C}\setminus\sigma(\mathcal{M}).
\]
We call $\mathcal{M}$ \emph{regular (of index 0)}, if there exists
$\nu\in\mathbb{R}$ such that:

\begin{enumerate}[(a)]

\item $\mathbb{C}_{\Re>\nu}\subseteq\varrho(\mathcal{M})$ and 

\item $\mathbb{C}_{\Re>\nu}\ni z\mapsto(zM_{0}+M_{1})^{-1}\in L(H)$
is bounded. 

\end{enumerate}

Moreover, we set 
\[
s_{0}(\mathcal{M})\coloneqq\inf\left\{ \nu\in\mathbb{R}\,;\,(a)\mbox{ and }(b)\mbox{ are satisfied}\right\} .
\]
\end{defn*}
\begin{example}
\label{exa:A-standard-example}A standard example for a regular linear
pencil is the following. Assume that $M_{0}$ is selfadjoint. Then
we can decompose the underlying space $H$ as $N(M_{0})\oplus\overline{R(M_{0})}$
by the projection theorem. Assume now that $M_{0}$ is strictly accretive
on $R(M_{0})$ and $M_{1}$ is strictly accretive on $N(M_{0})$,
i.e. there is $c>0$ such that 
\begin{align*}
\langle M_{0}x,x\rangle_{H} & \geq c|x|_{H}^{2}\\
\langle M_{1}y,y\rangle_{H} & \geq c|y|_{H}^{2}
\end{align*}
for each $x\in R(M_{0})$ and $y\in N(M_{0}).$ Then, the linear pencil
associated with $(M_{0},M_{1})$ is regular. Note that in the context
of partial differential equations the mentioned positive definiteness
conditions happen to be satisfied in many applications, see \cite[Section 3.3.3]{PicPhy}.
\end{example}

We note that in the latter example, the strict accretivity of $M_{0}$
on $R(M_{0})$ in particular implies that $R(M_{0})$ is closed. In
this case, we can characterise the regularity of the pencil by the
invertibility of a particular operator. For the next statement, we
introduce for a closed subspace $S\subseteq H$ the canonical embedding
\[
\iota_{S}\colon S\hookrightarrow H,s\mapsto s.
\]
It is easy to see that $\iota_{S}^{\ast}\colon H\to S$ is surjective
and acts as the orthogonal projection onto $S$, see also \cite[Lemma 3.2]{PTW15_FI}.
\begin{prop}
\label{prop:regPen}Assume that $R(M_{0})$ is closed and denote by
$\mathcal{M}$ the linear pencil associated with $(M_{0},M_{1}).$
Then the following statements are equivalent:

\begin{enumerate}[(i)]

\item $\mathcal{M}$ is regular,

\item $\iota_{R(M_{0})^{\bot}}^{\ast}M_{1}\iota_{N(M_{0})}:N(M_{0})\to R(M_{0})^{\bot}$
is continuously invertible.

\end{enumerate}

In the latter case, $\sigma(\mathcal{M})=\sigma\left(-\tilde{M}_{0}^{-1}\tilde{M}_{1}\right)$
is compact and for each $z\in\mathbb{C}$ we have that
\[
\mathcal{M}(z)=U_{1}^{\ast}V_{1}\left(\begin{array}{cc}
(z\tilde{M}_{0}+\tilde{M}_{1}) & 0\\
0 & \iota_{R(M_{0})^{\bot}}^{\ast}M_{1}\iota_{N(M_{0})}
\end{array}\right)V_{0}U_{0},
\]
where
\begin{align*}
U_{0}\coloneqq\left(\begin{array}{c}
\iota_{N(M_{0})^{\bot}}^{\ast}\\
\iota_{N(M_{0})}^{\ast}
\end{array}\right):H & \to N(M_{0})^{\bot}\oplus N(M_{0}),\\
U_{1}\coloneqq\left(\begin{array}{c}
\iota_{R(M_{0})}^{\ast}\\
\iota_{R(M_{0})^{\bot}}^{\ast}
\end{array}\right):H & \to R(M_{0})\oplus R(M_{0})^{\bot},
\end{align*}
and 
\begin{align*}
V_{0} & \coloneqq\left(\begin{array}{cc}
1 & 0\\
\left(\iota_{R(M_{0})^{\bot}}^{\ast}M_{1}\iota_{N(M_{0})}\right)^{-1}\left(\iota_{R(M_{0})^{\bot}}^{\ast}M_{1}\iota_{N(M_{0})^{\bot}}\right) & 1
\end{array}\right)\in L(N(M_{0})^{\bot}\oplus N(M_{0})),\\
V_{1} & \coloneqq\left(\begin{array}{cc}
1 & (\iota_{R(M_{0})}^{\ast}M_{1}\iota_{N(M_{0})})(\iota_{R(M_{0})^{\bot}}^{\ast}M_{1}\iota_{N(M_{0})})^{-1}\\
0 & 1
\end{array}\right)\in L(R(M_{0})\oplus R(M_{0})^{\bot}),
\end{align*}
as well as 
\begin{align*}
\tilde{M}_{0} & =\iota_{R(M_{0})}^{\ast}M_{0}\iota_{N(M_{0})^{\bot}},\\
\tilde{M}_{1} & =\iota_{R(M_{0})}^{\ast}M_{1}\iota_{N(M_{0})^{\bot}}\\
 & \quad-\left(\iota_{R(M_{0})}^{\ast}M_{1}\iota_{N(M_{0})}\right)\left(\iota_{R(M_{0})^{\bot}}^{\ast}M_{1}\iota_{N(M_{0})}\right)^{-1}\left(\iota_{R(M_{0})^{\bot}}^{\ast}M_{1}\iota_{N(M_{0})^{\bot}}\right).
\end{align*}
\end{prop}

\begin{proof}
(i) $\Rightarrow$ (ii): Assume that $\mathcal{M}$ is regular. By
the closed graph theorem, it suffices to prove that the operator $\iota_{R(M_{0})^{\bot}}^{\ast}M_{1}\iota_{N(M_{0})}$
is bijective. For showing that it is onto, let $f\in R(M_{0})^{\bot}.$
For $n\in\mathbb{N}$ sufficiently large we define
\[
u_{n}\coloneqq(nM_{0}+M_{1})^{-1}\iota_{R(M_{0})^{\bot}}f.
\]
Then, $(u_{n})_{n}$ is bounded and thus, by passing to a suitable
subsequence (not relabelled), we can assume that it is weakly convergent.
We denote its weak limit by $u$. Then 
\[
M_{0}u=\wlim_{n\to\infty}M_{0}u_{n}=\wlim_{n\to\infty}\frac{1}{n}(nM_{0}+M_{1})u_{n}=\wlim_{n\to\infty}\frac{1}{n}\iota_{R(M_{0})^{\bot}}f=0,
\]
i.e., $u\in N(M_{0}).$ Moreover, since 
\[
\iota_{R(M_{0})^{\bot}}^{\ast}M_{1}u_{n}=\iota_{R(M_{0})^{\bot}}^{\ast}(nM_{0}+M_{1})u_{n}=f
\]
for each $n\in\mathbb{N},$ we infer that 
\[
\iota_{R(M_{0})^{\bot}}^{\ast}M_{1}\iota_{N(M_{0})}u=f,
\]
which proves that $\iota_{R(M_{0})^{\bot}}^{\ast}M_{1}\iota_{N(M_{0})}$
is onto. To prove, that it is also one-to-one, let $u\in N(M_{0})$
with $\iota_{R(M_{0})^{\bot}}^{\ast}M_{1}\iota_{N(M_{0})}u=0.$ Consequently,
$M_{1}\iota_{N(M_{0})}u\in R(M_{0})$ and hence, there is $v\in H$
such that $M_{1}\iota_{N(M_{0})}u=M_{0}v.$ Since $(nM_{0}+M_{1})\iota_{N(M_{0})}u=M_{1}\iota_{N(M_{0})}u=M_{0}v$
for each $n\in\mathbb{N},$ we derive that 
\[
\iota_{N(M_{0})}u=(nM_{0}+M_{1})^{-1}M_{0}v=\frac{1}{n}v-\frac{1}{n}\left(nM_{0}+M_{1}\right)^{-1}M_{1}v
\]
for sufficiently large $n\in\mathbb{N}.$ As the right-hand side tends
to $0$ as $n$ tends to infinity, we infer $\iota_{N(M_{0})}u=0$
and thus, $u=0.$ 

(ii) $\Rightarrow$ (i): For $z\in\mathbb{C}$ we have that 
\begin{align*}
U_{1}\mathcal{M}(z)U_{0}^{\ast} & =\left(\begin{array}{cc}
\iota_{R(M_{0})}^{\ast}(zM_{0}+M_{1})\iota_{N(M_{0})^{\bot}} & \iota_{R(M_{0})}^{\ast}(zM_{0}+M_{1})\iota_{N(M_{0})}\\
\iota_{R(M_{0})^{\bot}}^{\ast}(zM_{0}+M_{1})\iota_{N(M_{0})^{\bot}} & \iota_{R(M_{0})^{\bot}}^{\ast}(zM_{0}+M_{1})\iota_{N(M_{0})}
\end{array}\right)\\
 & =\left(\begin{array}{cc}
\iota_{R(M_{0})}^{\ast}(zM_{0}+M_{1})\iota_{N(M_{0})^{\bot}} & \iota_{R(M_{0})}^{\ast}M_{1}\iota_{N(M_{0})}\\
\iota_{R(M_{0})^{\bot}}^{\ast}M_{1}\iota_{N(M_{0})^{\bot}} & \iota_{R(M_{0})^{\bot}}^{\ast}M_{1}\iota_{N(M_{0})}
\end{array}\right)\\
 & =V_{1}\left(\begin{array}{cc}
(z\tilde{M}_{0}+\tilde{M}_{1}) & 0\\
0 & \iota_{R(M_{0})^{\bot}}^{\ast}M_{1}\iota_{N(M_{0})}
\end{array}\right)V_{0}
\end{align*}
with $U_{0},U_{1},V_{0},V_{1},\tilde{M}_{0}$ and $\tilde{M}_{1}$
as above. Since $U_{0},U_{1}$ are unitary and $V_{0},V_{1}$ are
continuously invertible, we obtain using (ii) that $\mathcal{M}(z)$
is invertible, if and only if $(z\tilde{M}_{0}+\tilde{M}_{1})$ is
boundedly invertible. Moreover, there is a constant $C\geq0$ such
that for each $z\in\varrho(\mathcal{M})$ 
\[
\|\mathcal{M}(z)^{-1}\|\leq C\|(z\tilde{M}_{0}+\tilde{M}_{1})^{-1}\|.
\]
Since $\tilde{M}_{0}$ is bijective, it is boundedly invertible by
the closed graph theorem, and hence, 
\[
z\mapsto\left(z\tilde{M}_{0}+\tilde{M}_{1}\right)^{-1}=(z+\tilde{M}_{0}^{-1}\tilde{M}_{1})^{-1}\tilde{M}_{0}^{-1}
\]
is a well-defined bounded function on $\mathbb{C}_{\Re>\|\tilde{M}_{0}^{-1}\tilde{M}_{1}\|}$
by the Neumann series. In particular, $\sigma(\mathcal{M})=\sigma(-\tilde{M}_{0}^{-1}\tilde{M}_{1})$
is compact. 
\end{proof}
\begin{rem}
\label{rem:resolvent_bd}We note that from the representation of $\mathcal{M}(z)$
in \prettyref{prop:regPen} it follows that $\sup_{|z|>R}\|\mathcal{M}(z)^{-1}\|<\infty,$
where $R>0$ is sufficiently large. 
\end{rem}

\section{\label{sec:firstIV}A first initial value problem}

Throughout, let $M_{0},M_{1}\in L(H)$ such that $R(M_{0})$ is closed.
Moreover, we assume that the linear operator pencil $\mathcal{M}$
associated with $(M_{0},M_{1})$ is regular. In this section, we are
concerned with the following initial value problem 
\begin{align}
M_{0}u'(t)+M_{1}u(t) & =0\quad(t>0),\tag{IVP1}\label{eq:ivp_1}\\
u(0) & =u_{0},\nonumber 
\end{align}
where $u_{0}\in H$ is a given initial value and $u:\mathbb{R}_{\geq0}\to H$
is to be determined. We recall that \prettyref{eq:ivp_1} is a differential-algebraic
equation. The aim of this section is to find a characterisation of
possible initial values $u_{0}$ such that the solution of \prettyref{eq:ivp_1}
is \emph{continuously differentiable} on $\mathbb{R}_{>0}$. We present
a necessary condition for $u_{0}$ in the next statement. The main
result of this section will be to identify this necessary condition
also as sufficient, see \prettyref{cor:semigroupIVP} below. 
\begin{lem}
\label{lem:IV}If $u:\mathbb{R}_{\geq0}\to H$ is continuous, $u|_{\mathbb{R}_{>0}}$
is continuously differentiable and if $u$ satisfies \prettyref{eq:ivp_1},
then 
\[
u_{0}\in\mathrm{IV}\coloneqq\left\{ x\in H\,;\,M_{1}x\in R(M_{0})\right\} =M_{1}^{-1}\left[R(M_{0})\right].
\]
\end{lem}

\begin{proof}
Since we have 
\[
M_{1}u(t)=-M_{0}u'(t)\in R(M_{0})
\]
for each $t>0$, we infer 
\[
M_{1}u_{0}=M_{1}u(0)\in R(M_{0})
\]
due to the closedness of $R(M_{0}),$ i.e. $u_{0}\in\mathrm{IV}$.
\end{proof}
\begin{rem}
Note that $\mathrm{IV}$ is a closed subspace of $H,$ as it is the
pre-image of the closed subspace $R(M_{0})$ under $M_{1}.$
\end{rem}

For the converse of \prettyref{lem:IV}, we need a couple of preparations.
\begin{prop}
\label{prop:generator}The operator 
\[
\iota_{R(M_{0})}^{\ast}M_{0}\iota_{\mathrm{IV}}:\mathrm{IV}\to R(M_{0})
\]
is bijective and hence, continuously invertible.
\end{prop}

\begin{proof}
To show that the operator is one-to-one, let $x\in\mathrm{IV}\cap N(M_{0}).$
Hence, there is $y\in H$ such that $M_{1}x=M_{0}y$ and consequently,
\[
\left(nM_{0}+M_{1}\right)x=M_{1}x=M_{0}y
\]
for each $n\in\mathbb{N}.$ Thus, for $n$ large enough, we get 
\[
x=(nM_{0}+M_{1})^{-1}M_{0}y=\frac{1}{n}y-\frac{1}{n}(nM_{0}+M_{1})^{-1}M_{1}y
\]
and since $\limsup_{n\to\infty}\left\Vert (nM_{0}+M_{1})^{-1}\right\Vert <\infty$,
we infer $x=0.$ Thus, $\iota_{R(M_{0})}^{\ast}M_{0}\iota_{\mathrm{IV}}$
is one-to-one. For showing that it is onto, let $w\in R(M_{0}).$
For $n\in\mathbb{N}$ large enough, we set 
\[
u_{n}\coloneqq\left(nM_{0}+M_{1}\right)^{-1}nw.
\]
We claim that $u_{n}\in\mathrm{IV}.$ Indeed, using that 
\begin{align*}
M_{1}\left(nM_{0}+M_{1}\right)^{-1}M_{0} & =\frac{1}{n}M_{1}\left(1-(nM_{0}+M_{1})^{-1}M_{1}\right)\\
 & =\frac{1}{n}M_{1}-\frac{1}{n}(M_{1}-nM_{0}(nM_{0}+M_{1})^{-1}M_{1})\\
 & =M_{0}\left(nM_{0}+M_{1}\right)^{-1}M_{1},
\end{align*}
we infer that 
\[
(nM_{0}+M_{1})^{-1}[R(M_{0})]\subseteq\mathrm{IV}.
\]
Since $nw\in R(M_{0}),$ we derive the claim. Moreover, we find $y\in H$
such that $w=M_{0}y$ and we compute for all $n\in\mathbb{N}$ 
\[
u_{n}=(nM_{0}+M_{1})^{-1}nw=(nM_{0}+M_{1})^{-1}nM_{0}y=y-(nM_{0}+M_{1})^{-1}M_{1}y,
\]
which shows that $(u_{n})_{n\in\mathbb{N}}$ is a bounded sequence
in $\mathrm{IV}\subseteq H$ and hence, by passing to a suitable subsequence,
we may assume without loss of generality that $u_{n}\rightharpoonup u$
for some $u\in\mathrm{IV}.$ We obtain 
\[
M_{0}\iota_{\mathrm{IV}}u=\wlim_{n\to\infty}M_{0}u_{n}=\wlim_{n\to\infty}\frac{1}{n}(nM_{0}+M_{1})u_{n}=w,
\]
which proves the assertion. 
\end{proof}
\begin{rem}
\label{rem:IVcongNbot}We shall note here that $\textnormal{IV}$
and $N(M_{0})^{\bot}$ are isomorphic as Banach spaces. Indeed, $\iota_{R(M_{0})}^{\ast}M_{0}\iota_{N(M_{0})^{\bot}}$
is bijective and closed, and hence, a Banach space isomorphism. By
\prettyref{prop:generator}, $\iota_{R(M_{0})}^{\ast}M_{0}\iota_{\textnormal{IV}}$
is a Banach space isomorphism, as well. So,
\[
\left(\iota_{R(M_{0})}^{\ast}M_{0}\iota_{\textnormal{IV}}\right)^{-1}\iota_{R(M_{0})}^{\ast}M_{0}\iota_{N(M_{0})^{\bot}}\colon N(M_{0})^{\bot}\to\textnormal{IV}
\]
yields the desired isomorphism.
\end{rem}

\begin{cor}
\label{cor:semigroupIVP}Let $u_{0}\in\mathrm{IV}.$ Then there exists
a unique solution $u:\mathbb{R}_{\geq0}\to H$ of \prettyref{eq:ivp_1},
which is given by 
\[
u(t)=\exp\left(-t\left(\iota_{R(M_{0})}^{\ast}M_{0}\iota_{\mathrm{IV}}\right)^{-1}\left(\iota_{R(M_{0})}^{\ast}M_{1}\iota_{\mathrm{IV}}\right)\right)u_{0}\quad(t\geq0).
\]
\end{cor}

\begin{proof}
We address uniqueness first: If $u:\mathbb{R}_{\geq0}\to H$ is a
solution of \prettyref{eq:ivp_1}, then $M_{1}u(t)=-M_{0}u'(t)\in R(M_{0})$.
Hence, $u$ attains values in $\mathrm{IV}.$ Thus, equivalently,
\[
\iota_{R(M_{0})}^{\ast}M_{1}\iota_{\mathrm{IV}}u(t)=-\iota_{R(M_{0})}^{\ast}M_{0}\iota_{\mathrm{IV}}u'(t)
\]
and by \prettyref{prop:generator} 
\[
u'(t)=-\left(\iota_{R(M_{0})}^{\ast}M_{0}\iota_{\mathrm{IV}}\right)^{-1}\left(\iota_{R(M_{0})}^{\ast}M_{1}\iota_{\mathrm{IV}}\right)u(t)
\]
for $t>0.$ Hence, 
\[
u(t)=\exp\left(-t\left(\iota_{R(M_{0})}^{\ast}M_{0}\iota_{\mathrm{IV}}\right)^{-1}\left(\iota_{R(M_{0})}^{\ast}M_{1}\iota_{\mathrm{IV}}\right)\right)u_{0}\quad(t\geq0).
\]
For the existence part consider 
\begin{align*}
\mathbb{R}_{\geq0} & \ni t\mapsto u(t)=\exp\left(-t\left(\iota_{R(M_{0})}^{\ast}M_{0}\iota_{\mathrm{IV}}\right)^{-1}\left(\iota_{R(M_{0})}^{\ast}M_{1}\iota_{\mathrm{IV}}\right)\right)u_{0}.
\end{align*}
Then clearly $u(0)=u_{0}$ and 
\[
u'(t)=-\left(\iota_{R(M_{0})}^{\ast}M_{0}\iota_{\mathrm{IV}}\right)^{-1}\left(\iota_{R(M_{0})}^{\ast}M_{1}\iota_{\mathrm{IV}}\right)u(t)
\]
which in turn implies 
\[
M_{0}u'(t)+M_{1}u(t)=0
\]
for each $t>0.$ Hence, $u$ solves \prettyref{eq:ivp_1}.
\end{proof}

\section{\label{sec:secondIV}A second initial value problem}

In this section, we aim at providing a different perspective to the
initial value problem stated in \prettyref{eq:ivp_1}. In the previous
section, we have addressed finding solutions $u$ in the ``strong
sense'', that is, we were looking for continuously differentiable
solutions. The differentiability, in turn, restricted the class of
admissible initial values. Here, we study the solvability of the differential-algebraic
equation in a ``weak'' or ``mild'' sense. Weakening the solution
concepts, we will be able to solve the differential-algebraic equation
for \emph{all} initial data from $H$. For this, we provide a different
solution representation of the solution to \prettyref{eq:ivp_1}. 

Any function defined on $\mathbb{R}_{\geq0}$ is considered to be
a function on $\mathbb{R}$ by extension by $0$. The support of a
function $f$ will be denoted by $\spt f$.
\begin{prop}
\label{prop:LaplaceRep}Let $u_{0}\in\mathrm{IV}$ and $u:\mathbb{R}_{\geq0}\to H$
the solution of \prettyref{eq:ivp_1}. Then there is $\rho_{0}\geq0$
such that $u\in\bigcap_{\rho>\rho_{0}}L_{2,\rho}(\mathbb{R};H)$ and
\[
\left(\mathcal{L}_{\rho}u\right)(t)=\frac{1}{\sqrt{2\pi}}\left((\i t+\rho)M_{0}+M_{1}\right)^{-1}M_{0}u_{0}\quad(\rho>\rho_{0},t\in\mathbb{R}).
\]
\end{prop}

\begin{proof}
Since by \prettyref{cor:semigroupIVP}
\[
u(t)=\exp\left(-t\left(\iota_{R(M_{0})}^{\ast}M_{0}\iota_{\mathrm{IV}}\right)^{-1}\left(\iota_{R(M_{0})}^{\ast}M_{1}\iota_{\mathrm{IV}}\right)\right)u_{0}
\]
for $t\geq0$, we infer that $|u(t)|\leq M\e^{\rho_{1}t}$ for some
$M\geq1,\rho_{1}\in\mathbb{R}$ and all $t\geq0.$ Hence, using that
$\spt u\subseteq\mathbb{R}_{\geq0}$, we obtain $u\in\bigcap_{\rho>\rho_{1}}L_{2,\rho}(\mathbb{R};H).$
We define $\rho_{0}\coloneqq\max\{0,\rho_{1},s_{0}(\mathcal{M})\}$
and claim that 
\[
u-\chi_{\mathbb{R}_{\geq0}}u_{0}\in\bigcap_{\rho>\rho_{0}}H_{\rho}^{1}(\mathbb{R};H).
\]
Indeed, for $\rho>\rho_{0}$ we have that $u-\chi_{\mathbb{R}_{\geq0}}u_{0}\in L_{2,\rho}(\mathbb{R};H)$
and for $\varphi\in C_{c}^{\infty}(\mathbb{R};H)$ we compute using
\prettyref{cor:spadj} 
\begin{align*}
\langle u-\chi_{\mathbb{R}_{\geq0}}u_{0},\partial_{0,\rho}^{\ast}\varphi\rangle_{\rho} & =\intop_{0}^{\infty}\langle u(t)-u_{0},-\varphi'(t)+2\rho\varphi(t)\rangle_{H}\e^{-2\rho t}\mbox{ d}t\\
 & =\intop_{0}^{\infty}\langle u'(t),\varphi(t)\rangle_{H}\e^{-2\rho t}\mbox{ d}t\\
 & =\langle u',\varphi\rangle_{\rho}
\end{align*}
by integration by parts, and thus $u-\chi_{\mathbb{R}_{\geq0}}u_{0}\in H_{\rho}^{1}(\mathbb{R};H)$
with $\partial_{0,\rho}(u-\chi_{\mathbb{R}_{\geq0}}u_{0})=u'.$ Hence,
we compute with \prettyref{prop:FLtransform} 
\begin{align*}
 & \left((\i t+\rho)M_{0}+M_{1}\right)\left(\mathcal{L}_{\rho}u\right)(t)\\
 & =(\i t+\rho)M_{0}\left(\mathcal{L}_{\rho}(u-\chi_{\mathbb{R}_{\geq0}}u_{0})\right)(t)+M_{1}\left(\mathcal{L}_{\rho}u\right)(t)+(\i t+\rho)M_{0}\left(\mathcal{L}_{\rho}\chi_{\mathbb{R}_{\geq0}}u_{0}\right)(t)\\
 & =M_{0}\mathcal{L}_{\rho}\left(\partial_{0,\rho}(u-\chi_{\mathbb{R}_{\geq0}}u_{0})\right)(t)+M_{1}\left(\mathcal{L}_{\rho}u\right)(t)+\frac{1}{\sqrt{2\pi}}M_{0}u_{0}\\
 & =\mathcal{L}_{\rho}\left(M_{0}u'+M_{1}u\right)(t)+\frac{1}{\sqrt{2\pi}}M_{0}u_{0}\\
 & =\frac{1}{\sqrt{2\pi}}M_{0}u_{0},
\end{align*}
from which we read off 
\[
\left(\mathcal{L}_{\rho}u\right)(t)=\frac{1}{\sqrt{2\pi}}\left((\i t+\rho)M_{0}+M_{1}\right)^{-1}M_{0}u_{0}
\]
for each $t\in\mathbb{R}.$ 
\end{proof}
As we have seen in \prettyref{prop:LaplaceRep}, the solution $u$
of \prettyref{eq:ivp_1} is given by the relation 
\[
\left(\mathcal{L}_{\rho}u\right)(t)=\frac{1}{\sqrt{2\pi}}\left((\i t+\rho)M_{0}+M_{1}\right)^{-1}M_{0}u_{0}\quad(t\in\mathbb{R})
\]
for some $\rho>0$ large enough and $u_{0}\in\mathrm{IV}.$ Thus,
we obtain the following consequence of \prettyref{cor:semigroupIVP}
and \prettyref{prop:LaplaceRep}:
\begin{cor}
\label{cor:IVsemLapl} Let $u_{0}\in\mathrm{IV}$. Then for $\rho>0$
large enough we get
\begin{multline*}
\chi_{\mathbb{R}_{\geq0}}(t)\exp\left(-t\left(\iota_{R(M_{0})}^{\ast}M_{0}\iota_{\mathrm{IV}}\right)^{-1}\iota_{R(M_{0})}^{\ast}M_{1}\iota_{\mathrm{IV}}\right)u_{0}\\
=\frac{1}{\sqrt{2\pi}}\left(\mathcal{L}_{\rho}^{\ast}\left((\i\cdot+\rho)M_{0}+M_{1}\right)^{-1}M_{0}u_{0}\right)(t)\quad(t\in\mathbb{R}).
\end{multline*}
 
\end{cor}

Note that the right-hand side of the equation in \prettyref{cor:IVsemLapl}
also makes sense for $u_{0}\in H$. Indeed, if we just require $u_{0}\in H,$
then $u$ given as above solves a weaker variant of the initial value
problem \prettyref{eq:ivp_1}: We introduce the respective notion
first:
\begin{defn}
Let $u_{0}\in H$, and $u\colon\mathbb{R}_{\geq0}\to H$ locally integrable.
Then $u$ is a \emph{mild solution of \prettyref{eq:ivp_1},} if $u$
satisfies 
\begin{align}
\left(M_{0}u\right)(t)+\int_{0}^{t}M_{1}u(s)\,\mathrm{d}s & =M_{0}u_{0}\quad(t>0),\nonumber \\
\left(M_{0}u\right)(0+) & =M_{0}u_{0}.\tag{{IVP2}}\label{eq:mIVP}
\end{align}
Note that, if $u$ is a solution to \prettyref{eq:mIVP}, then it
follows from Lebesgue's dominated convergence theorem that $M_{0}u\colon\mathbb{R}_{\geq0}\to H$
is continuous and that the initial datum is attained. 
\end{defn}

\begin{thm}
\label{thm:IVP_2}Let $u\in L_{1,\mathrm{loc}}(\mathbb{R};H)$. Then
the following are equivalent

\begin{enumerate}[(i)]

\item $\spt u\subseteq\mathbb{R}_{\geq0}$ and $u$ satisfies 
\begin{align*}
(M_{0}u)'(t)+M_{1}u(t) & =0\quad(t>0\mbox{ a.e.}),\\
\lim_{t\to0+}\frac{1}{t}\intop_{0}^{t}|M_{0}\left(u(s)-u_{0}\right)|\:\mathrm{d}s & =0
\end{align*}
in the sense of distributions, i.e. 
\[
-\intop_{0}^{\infty}\langle M_{0}u(t),\varphi'(t)\rangle_{H}\,\mathrm{d}t+\intop_{0}^{\infty}\langle M_{1}u(t),\varphi(t)\rangle_{H}\mbox{ d}t=0
\]
for each $\varphi\in C_{c}^{\infty}(\mathbb{R}_{>0};H).$

\item $\spt u\subseteq\mathbb{R}_{\geq0}$ and 
\begin{equation}
M_{0}u(t)+\intop_{0}^{t}M_{1}u(s)\,\mathrm{d}s=M_{0}u_{0}\quad(t>0\mbox{ a.e.})\label{eq:IVP_2}
\end{equation}
i.e., $u$ is a mild solution of \prettyref{eq:ivp_1}.

\item For $\rho>\max\left\{ s_{0}(\mathcal{M}),0\right\} $ we have
$u\in L_{2,\rho}(\mathbb{R};H)$ and 
\[
\left(\mathcal{L}_{\rho}u\right)(t)=\frac{1}{\sqrt{2\pi}}\left((\i t+\rho)M_{0}+M_{1}\right)^{-1}M_{0}u_{0}\quad(t\in\mathbb{R}).
\]

\end{enumerate}

In the latter case $M_{0}(u-\chi_{\mathbb{R}_{\geq0}}u_{0})\in H_{\rho}^{1}(\mathbb{R};H)$
for each $\rho>\max\left\{ s_{0}(\mathcal{M}),0\right\} .$
\end{thm}

Note that, by \prettyref{thm:IVP_2}, we deduce that if $u$ is a
mild solution of \prettyref{eq:ivp_1}, then $u\in L_{2,\rho}(\mathbb{R};H)$.
Moreover, we obtain that $M_{0}(u-\chi_{\mathbb{R}_{\geq0}}u_{0})\in H_{\rho}^{1}(\mathbb{R};H)$
with
\[
\partial_{0,\rho}M_{0}(u-\chi_{\mathbb{R}_{\geq0}}u_{0})+M_{1}u=0
\]
 valid in $L_{2,\rho}(\mathbb{R};H)$ for some $\rho>0.$ Since the
pencil $\mathcal{M}$ associated with $(M_{0},M_{1})$ is regular,
we obtain $\left(\partial_{0,\rho}M_{0}+M_{1}\right)^{-1}\in L(L_{2,\rho}(\mathbb{R};H))$
is well-defined (see e.g.~\cite[Lemma 2.2]{Trostorff2015a}) and
\[
u=\chi_{\mathbb{R}_{\geq0}}u_{0}-\left(\partial_{0,\rho}M_{0}+M_{1}\right)^{-1}M_{1}\chi_{\mathbb{R}_{\geq0}}u_{0},
\]
which yields another representation for the solution $u$.

Before we can come to the proof of \prettyref{thm:IVP_2}, we need
some auxiliary results first. These statements focus on relating properties
of holomorphic functions to the support of their Fourier\textendash Laplace
transform. In this sense, one can think of the subsequent assertions
as variant of the Paley\textendash Wiener theorem. For this we need
a prerequisite:
\begin{lem}[{see also \cite[Lemma 3.6]{Trostorff2013}}]
\label{lem:independent_rho}Let $\mu,\rho\in\mathbb{R}$ with $\mu<\rho$
and let $U\coloneqq\left\{ z\in\mathbb{C}\,;\,\mu<\Re z<\rho\right\} .$
Moreover, let $f:\overline{U}\to H$ be continuous and analytic in
$U$ with 
\[
M\coloneqq\sup_{z\in\overline{U}}|zf(z)|_{H}<\infty.
\]
Then $f(\i\cdot+\kappa)\in L_{2}(\mathbb{R};H)$ for each $\mu\leq\kappa\leq\rho$
and 
\[
\left(\mathcal{L}_{\rho}^{\ast}f(\i\cdot+\rho)\right)(t)=\left(\mathcal{L}_{\mu}^{\ast}f(\i\cdot+\mu)\right)(t)\quad(t\in\mathbb{R}\mbox{ a.e.}).
\]
\end{lem}

\begin{proof}
The first claim follows from the boundedness of $t\mapsto(\i t+\kappa)f(\i t+\kappa)$
for each $\mu\leq\kappa\leq\rho$. For showing the second claim, let
$\left(R_{n}\right)_{n\in\mathbb{N}}$ be a sequence in $\mathbb{R}_{>0}$
such that $R_{n}\to\infty$ and 
\begin{align}
\left(\mathcal{L}_{\rho}^{\ast}\chi_{[-R_{n},R_{n}]}f(\i\cdot+\rho)\right)(t) & \to\left(\mathcal{L}_{\rho}^{\ast}f(\i\cdot+\rho)\right)(t)\label{eq:appr_almost_everywhere}\\
\left(\mathcal{L}_{\mu}^{\ast}\chi_{[-R_{n},R_{n}]}f(\i\cdot+\mu)\right)(t) & \to\left(\mathcal{L}_{\mu}^{\ast}f(\i\cdot+\mu)\right)(t)\nonumber 
\end{align}
for almost every $t\in\mathbb{R}$ as $n\to\infty$. For $n\in\mathbb{N},t\in\mathbb{R}$
we compute 
\begin{align*}
 & \left|\left(\mathcal{L}_{\rho}^{\ast}\chi_{[-R_{n},R_{n}]}f(\i\cdot+\rho)\right)(t)-\left(\mathcal{L}_{\mu}^{\ast}\chi_{[-R_{n},R_{n}]}f(\i\cdot+\mu)\right)(t)\right|_{H}\\
 & =\frac{1}{\sqrt{2\pi}}\left|\intop_{-R_{n}}^{R_{n}}\e^{\left(\i s+\rho\right)t}f(\i s+\rho)\mbox{ d}s-\intop_{-R_{n}}^{R_{n}}\e^{(\i s+\mu)t}f(\i s+\mu)\mbox{ d}s\right|_{H}\\
 & =\frac{1}{\sqrt{2\pi}}\left|\intop_{\mu}^{\rho}\e^{\left(\i R_{n}+\kappa\right)t}f(\i R_{n}+\kappa)\mbox{ d}\kappa-\intop_{\mu}^{\rho}\e^{(-\i R_{n}+\kappa)t}f(-\i R_{n}+\kappa)\mbox{ d}\kappa\right|_{H}
\end{align*}
by the Cauchy integral theorem. Hence, we can estimate 
\begin{align*}
 & \left|\left(\mathcal{L}_{\rho}^{\ast}\chi_{[-R_{n},R_{n}]}f(\i\cdot+\rho)\right)(t)-\left(\mathcal{L}_{\mu}^{\ast}\chi_{[-R_{n},R_{n}]}f(\i\cdot+\mu)\right)(t)\right|_{H}\\
 & \leq\frac{2M}{\sqrt{2\pi}}\intop_{\mu}^{\rho}\e^{\kappa t}\frac{1}{|\i R_{n}+\kappa|}\mbox{ d}\kappa\to0\quad\left(n\to\infty\right).
\end{align*}
Together with \prettyref{eq:appr_almost_everywhere}, this yields
the assertion.
\end{proof}
\begin{lem}
\label{lem:causality_1}Let $\rho_{0}\in\mathbb{R}$ and $f:\mathbb{C}_{\Re\geq\rho_{0}}\to H$
continuous, analytic in $\mathbb{C}_{\Re>\rho_{0}}$ with
\[
M\coloneqq\sup_{z\in\mathbb{C}_{\Re\geq\rho_{0}}}|zf(z)|_{H}<\infty.
\]
Then $\spt\mathcal{L}_{\rho_{0}}^{\ast}f(\i\cdot+\rho_{0})\subseteq\mathbb{R}_{\geq0}$.
\end{lem}

\begin{proof}
Let $\varphi\in C_{c}^{\infty}(\mathbb{R}_{<0};H)$ and $\rho>\max\{\rho_{0},0\}$.
Using \prettyref{lem:independent_rho} we obtain 
\begin{align*}
 & \left|\intop_{\mathbb{R}}\langle\left(\mathcal{L}_{\rho_{0}}^{\ast}f(\i\cdot+\rho_{0})\right)(t),\varphi(t)\rangle_{H}\mbox{ d}t\right|\\
 & =\left|\intop_{\mathbb{R}}\langle\left(\mathcal{L}_{\rho}^{\ast}f(\i\cdot+\rho)\right)(t),\varphi(t)\e^{2\rho t}\rangle_{H}\e^{-2\rho t}\mbox{ d}t\right|\\
 & \leq|\mathcal{L}_{\rho}^{\ast}f(\i\cdot+\rho)|_{\rho}|\varphi\e^{2\rho(\cdot)}|_{\rho}\\
 & =|f(\i\cdot+\rho)|_{0}\left(\intop_{-\infty}^{0}|\varphi(t)\e^{2\rho t}|_{H}^{2}\e^{-2\rho t}\mbox{ d}t\right)^{\frac{1}{2}}\\
 & \leq M\left|\frac{1}{\i\cdot+\rho}\right|_{0}|\varphi|_{0}\\
 & =M\sqrt{\frac{\pi}{\rho}}|\varphi|_{0}.
\end{align*}
Letting $\rho$ tend to infinity, we infer 
\[
\intop_{\mathbb{R}}\langle\left(\mathcal{L}_{\rho_{0}}^{\ast}f(\i\cdot+\rho_{0})\right)(t),\varphi(t)\rangle_{H}\mbox{ d}t=0
\]
which yields the assertion. 
\end{proof}
\begin{rem}
The latter lemma is a special case of the famous Paley-Wiener Theorem
(see \cite{Paley1987} or \cite[19.2 Theorem]{rudin1987real}), characterising
the $L_{2}$ functions supported on the positive real line by their
Laplace transform lying in the Hardy space $\mathcal{H}^{2}(\mathbb{C}_{\Re>0};H).$
However, to keep this article self-contained and since we do not need
this deeper result here, we restrict ourselves to the easier case
outlined in \prettyref{lem:causality_1}.
\end{rem}

\begin{cor}
\label{cor:causlity_2}Let $\rho_{0}\in\mathbb{R}$, $F:\mathbb{C}_{\Re\geq\rho_{0}}\to L(H)$
bounded, continuous and analytic in $\mathbb{C}_{\Re>\rho_{0}}.$
Let $g\in L_{2,\rho_{0}}(\mathbb{R};H)$ with $\spt g\subseteq\mathbb{R}_{\geq a}$
for some $a\in\mathbb{R}.$ Then 
\[
\spt\mathcal{L}_{\rho_{0}}^{\ast}F(\i\cdot+\rho_{0})\mathcal{L}_{\rho_{0}}g\subseteq\mathbb{R}_{\geq a}.
\]
\end{cor}

\begin{proof}
Let $(\phi_{n})_{n\in\mathbb{N}}$ be a sequence in $C_{c}^{\infty}(\mathbb{R};H)$
such that $\spt\phi_{n}\subseteq\mathbb{R}_{\geq a}$ for all $n\in\mathbb{N}$
and $\phi_{n}\to g$ in $L_{2,\rho_{0}}(\mathbb{R};H)$ as $n\to\infty.$
By continuity 
\[
\mathcal{L}_{\rho_{0}}^{\ast}F(\i\cdot+\rho_{0})\mathcal{L}_{\rho_{0}}\phi_{n}\to\mathcal{L}_{\rho_{0}}^{\ast}F(\i\cdot+\rho_{0})\mathcal{L}_{\rho_{0}}g\mbox{ in }L_{2,\rho_{0}}(\mathbb{R};H)\quad(n\to\infty)
\]
and hence, it suffices to show that 
\[
\spt\mathcal{L}_{\rho_{0}}^{\ast}F(\i\cdot+\rho_{0})\mathcal{L}_{\rho_{0}}\psi\subseteq\mathbb{R}_{\geq a}
\]
for $\psi\in C_{c}^{\infty}(\mathbb{R}_{\geq a};H).$ Next, we shall
argue that we may assume without loss of generality $a=0$. Indeed,
let $\psi\in C_{c}^{\infty}(\mathbb{R}_{\geq a};H)$ and consider
the function $\tilde{\psi}\coloneqq\psi(\cdot+a)\in C_{c}^{\infty}(\mathbb{R}_{\geq0};H).$
Then we have 
\[
\mathcal{L}_{\rho_{0}}\tilde{\psi}=\e^{(-\i(\cdot)+\rho_{0})a}\mathcal{L}_{\rho_{0}}\psi
\]
and consequently 
\[
\left(\mathcal{L}_{\rho_{0}}^{\ast}F(\i\cdot+\rho_{0})\mathcal{L}_{\rho_{0}}\tilde{\psi}\right)=\left(\mathcal{L}_{\rho_{0}}^{\ast}F(\i\cdot+\rho_{0})\mathcal{L}_{\rho_{0}}\psi\right)(\cdot+a).
\]
Thus, we need to check 
\[
\spt\mathcal{L}_{\rho_{0}}^{\ast}F(\i\cdot+\rho_{0})\mathcal{L}_{\rho_{0}}\tilde{\psi}\subseteq\mathbb{R}_{\geq0}.
\]
Hence, proving the case $a=0$ indeed implies the general statement.
So, consider the function $f:\mathbb{C}_{\Re\geq\rho_{0}}\to H,$
given by 
\[
f(z)=F(z)\hat{\psi}(z),
\]
where 
\[
\hat{\psi}(z)=\frac{1}{\sqrt{2\pi}}\intop_{\mathbb{R}}\e^{-zs}\psi(s)\mbox{ d}s=\left(\mathcal{L}_{\Re z}\psi\right)(\Im z).
\]
Then clearly, $f$ is continuous and analytic in $\mathbb{C}_{\Re>\rho_{0}}$
and 
\[
|zf(z)|_{H}=\left|F(z)\hat{\psi}'(z)\right|_{H}\leq\|F\|_{\infty}\|\psi'\|_{\infty}\frac{1}{\sqrt{2\pi}}\intop_{0}^{\max\spt\psi}\e^{-\rho_{0}s}\mbox{ d}s.
\]
Hence, 
\[
\spt\mathcal{L}_{\rho_{0}}^{\ast}F(\i\cdot+\rho_{0})\mathcal{L}_{\rho_{0}}\tilde{\psi}=\spt\mathcal{L}_{\rho_{0}}^{\ast}f(\i\cdot+\rho_{0})\subseteq\mathbb{R}_{\geq0}
\]
by \prettyref{lem:causality_1}.
\end{proof}
Having these prerequisites at hand, we are now in the position to
prove the main result of this section.
\begin{proof}[Proof of \prettyref{thm:IVP_2}.]
(i) $\Rightarrow$ (ii): Let $\psi\in C_{c}^{\infty}(\mathbb{R}_{\geq0};H)$.
We compute
\begin{multline*}
\intop_{0}^{\infty}\left\langle M_{0}u(t)+\intop_{0}^{t}M_{1}u(s)\mbox{ d}s,\psi(t)\right\rangle _{H}\mbox{ d}t\\
=\intop_{0}^{\infty}\langle M_{0}u(t),\psi(t)\rangle_{H}\mbox{ d}t+\intop_{0}^{\infty}\left\langle M_{1}u(t),\intop_{t}^{\infty}\psi(s)\mbox{ d}s\right\rangle _{H}\mbox{ d}t.
\end{multline*}
Setting $\varphi(t)\coloneqq\intop_{t}^{\infty}\psi(s)\mbox{ d}s$
for $t\in\mathbb{R}_{\geq0}$, the latter can be expressed by 
\begin{multline*}
\intop_{0}^{\infty}\left\langle M_{0}u(t)+\intop_{0}^{t}M_{1}u(s)\mbox{ d}s,\psi(t)\right\rangle _{H}\mbox{ d}t\\
=-\intop_{0}^{\infty}\langle M_{0}u(t),\varphi'(t)\rangle_{H}\mbox{ d}t+\intop_{0}^{\infty}\langle M_{1}u(t),\varphi(t)\rangle_{H}\mbox{ d}t.
\end{multline*}
Let now $\eta_{n}\in C^{\infty}(\mathbb{R})$ such that $\spt\eta_{n}\subseteq[\frac{1}{n},\infty[,$
$\eta_{n}=1$ on $[\frac{2}{n},\infty[$ and $|\eta_{n}|_{\infty}\leq1,|\eta_{n}'|_{\infty}\leq2n$
for each $n\in\mathbb{N}$. Set $\varphi_{n}\coloneqq\varphi\eta_{n}\in C_{c}^{\infty}(\mathbb{R}_{>0};H)$
for $n\in\mathbb{N}$. We obtain
\begin{align*}
 & -\intop_{0}^{\infty}\langle M_{0}u(t),\varphi'(t)\rangle_{H}\mbox{ d}t+\intop_{0}^{\infty}\langle M_{1}u(t),\varphi(t)\rangle_{H}\mbox{ d}t\\
= & \lim_{n\to\infty}\left(-\intop_{0}^{\infty}\langle M_{0}u(t),\varphi'(t)\eta_{n}(t)\rangle_{H}\mbox{ d}t+\intop_{0}^{\infty}\langle M_{1}u(t),\varphi_{n}(t)\rangle_{H}\mbox{ d}t\right)\\
= & \lim_{n\to\infty}\left(-\intop_{0}^{\infty}\langle M_{0}u(t),\varphi_{n}'(t)\rangle_{H}\mbox{ d}t+\intop_{0}^{\infty}\left\langle M_{0}u(t),\varphi(t)\eta_{n}'(t)\right\rangle _{H}\mbox{ d}t+\intop_{0}^{\infty}\langle M_{1}u(t),\varphi_{n}(t)\rangle_{H}\mbox{ d}t\right)\\
= & \lim_{n\to\infty}\intop_{0}^{\infty}\left\langle M_{0}u(t),\varphi(t)\eta_{n}'(t)\right\rangle _{H}\mbox{ d}t.
\end{align*}
Moreover, we have that 
\[
\left|\intop_{0}^{\infty}\left\langle M_{0}\left(u(t)-u_{0}\right),\varphi(t)\eta_{n}'(t)\right\rangle _{H}\mbox{ d}t\right|\leq2n\|\varphi\|_{\infty}\intop_{0}^{\frac{2}{n}}|M_{0}\left(u(t)-u_{0}\right)|_{H}\mbox{ d}t\to0\quad(n\to\infty)
\]
and thus, 
\begin{align*}
 & \lim_{n\to\infty}\intop_{0}^{\infty}\left\langle M_{0}u(t),\varphi(t)\eta_{n}'(t)\right\rangle _{H}\mbox{ d}t\\
 & =\lim_{n\to\infty}\intop_{0}^{\infty}\langle M_{0}u_{0},\varphi(t)\eta_{n}'(t)\rangle_{H}\mbox{ d}t\\
 & =\lim_{n\to\infty}\left(\intop_{0}^{\infty}\langle M_{0}u_{0},\varphi_{n}'(t)\rangle_{H}\mbox{ d}t-\intop_{0}^{\infty}\langle M_{0}u_{0},\varphi'(t)\eta_{n}(t)\rangle_{H}\mbox{ d}t\right)\\
 & =-\intop_{0}^{\infty}\langle M_{0}u_{0},\varphi'(t)\rangle_{H}\mbox{ d}t\\
 & =\intop_{0}^{\infty}\langle M_{0}u_{0},\psi(t)\rangle_{H}\mbox{ d}t.
\end{align*}
Summarising, we have shown that 
\[
\intop_{0}^{\infty}\left\langle M_{0}u(t)+\intop_{0}^{t}M_{1}u(s)\mbox{ d}s,\psi(t)\right\rangle _{H}\mbox{ d}t=\intop_{0}^{\infty}\langle M_{0}u_{0},\psi(t)\rangle_{H}\mbox{ d}t
\]
for each $\psi\in C_{c}^{\infty}(\mathbb{R}_{\geq0};H)$ and hence,
the assertion follows.\\
(ii) $\Rightarrow$ (iii): Let $\varphi\in C_{c}^{\infty}(\mathbb{R};H)$
and $\rho>\max\{s_{0}(\mathcal{M}),0\}$. Denote by $a\coloneqq\max\spt\varphi.$
We set 
\[
\psi(t)\coloneqq(-\i t+\rho)\left((-\i t+\rho)M_{0}^{\ast}+M_{1}^{\ast}\right)^{-1}(\mathcal{L}_{\rho}\varphi)(t)\quad(t\in\mathbb{R}).
\]
Then we infer 
\begin{align}
\intop_{\mathbb{R}}\langle u(t),\varphi(t)\rangle_{H}\e^{-2\rho t}\mbox{ d}t & =\intop_{0}^{a}\langle u(t),\varphi(t)\rangle_{H}\e^{-2\rho t}\mbox{ d}t\nonumber \\
 & =\intop_{0}^{a}\left\langle u(t),\left(\mathcal{L}_{\rho}^{\ast}\left(M_{0}^{\ast}+\frac{1}{-\i\cdot+\rho}M_{1}^{\ast}\right)\psi\right)(t)\right\rangle _{H}\e^{-2\rho t}\mbox{ d}t.\label{eq:e1}
\end{align}
As a first step we show that $\spt\mathcal{L}_{\rho}^{\ast}\psi\subseteq]-\infty,a].$
For doing so, let $\xi\in C_{c}^{\infty}(\mathbb{R}_{>a},H).$ Then
we compute 
\begin{align*}
\intop_{\mathbb{R}}\langle\mathcal{L}_{\rho}^{\ast}\psi(t),\xi(t)\rangle_{H}\e^{-2\rho t}\mbox{ d}t & =\intop_{\mathbb{R}}\langle\psi(t),\mathcal{L}_{\rho}\xi(t)\rangle_{H}\mbox{ d}t\\
 & =\intop_{\mathbb{R}}\left\langle \left(\mathcal{L}_{\rho}\varphi\right)(t),\left((\i t+\rho)M_{0}+M_{1}\right)^{-1}\left(\mathcal{L}_{\rho}\xi'\right)(t)\right\rangle _{H}\mbox{ d}t\\
 & =\intop_{-\infty}^{a}\langle\varphi(t),\left(\mathcal{L}_{\rho}^{\ast}\left(\left(\i\cdot+\rho\right)M_{0}+M_{1}\right)^{-1}\mathcal{L}_{\rho}\xi'\right)(t)\rangle_{H}\e^{-2\rho t}\mbox{ d}t,
\end{align*}
where in the last equality we used $\spt\varphi\subseteq]-\infty,a]$.
Next, by \prettyref{cor:causlity_2} applied to $g=\xi'$ and $F\colon z\mapsto\left(zM_{0}+M_{1}\right)^{-1}$,
we deduce that $\spt\xi'\subseteq\mathbb{R}_{\geq a}$ implies $\spt\left(\mathcal{L}_{\rho}^{\ast}F(\i\cdot+\rho)\mathcal{L}_{\rho}\xi'\right)\subseteq\mathbb{R}_{\geq a}$.
Hence,
\[
\intop_{\mathbb{R}}\langle\mathcal{L}_{\rho}^{\ast}\psi(t),\xi(t)\rangle_{H}\e^{-2\rho t}\mbox{ d}t=0.
\]
Since $\xi$ was arbitrary, we conclude the first step and deduce
that
\begin{equation}
\spt\mathcal{L}_{\rho}^{*}\psi\subseteq]-\infty,a].\label{eq:Lpsisup}
\end{equation}
 Moreover, we compute 
\begin{align}
\mathcal{L}_{\rho}^{\ast}\left(\frac{1}{-\i\cdot+\rho}\psi\right)(t)\e^{-2\rho t} & =\frac{1}{\sqrt{2\pi}}\intop_{\mathbb{R}}\frac{1}{-\i s+\rho}\psi(s)\e^{\left(\i s-\rho\right)t}\mbox{ d}s\nonumber \\
 & =\frac{1}{\sqrt{2\pi}}\intop_{\mathbb{R}}\intop_{t}^{\infty}\e^{(\i s-\rho)r}\mbox{ d}r\psi(s)\mbox{ d}s\nonumber \\
 & =\intop_{t}^{\infty}\e^{-\rho r}\frac{1}{\sqrt{2\pi}}\intop_{\mathbb{R}}\e^{\i sr}\psi(s)\mbox{ d}s\mbox{ d}r\nonumber \\
 & =\intop_{t}^{a}\left(\mathcal{L}_{\rho}^{\ast}\psi\right)(r)\e^{-2\rho r}\mbox{ d}r\label{eq:e2}
\end{align}
for each $t\in\mathbb{R}$ and hence, with \prettyref{eq:e1} and
\prettyref{eq:e2} we get 
\begin{align*}
 & \intop_{\mathbb{R}}\langle u(t),\varphi(t)\rangle_{H}\e^{-2\rho t}\mbox{ d}t\\
 & =\intop_{0}^{a}\left\langle u(t),M_{0}^{\ast}\left(\mathcal{L}_{\rho}^{\ast}\psi\right)(t)\e^{-2\rho t}+M_{1}^{\ast}\intop_{t}^{a}\left(\mathcal{L}_{\rho}^{\ast}\psi\right)(r)\e^{-2\rho r}\mbox{ d}r\right\rangle _{H}\mbox{ d}t\\
 & =\intop_{0}^{a}\langle M_{0}u(t),\left(\mathcal{L}_{\rho}^{\ast}\psi\right)(t)\rangle_{H}\e^{-2\rho t}\mbox{ d}t+\intop_{0}^{a}\intop_{t}^{a}\langle M_{1}u(t),\left(\mathcal{L}_{\rho}^{\ast}\psi\right)(r)\rangle_{H}\e^{-2\rho r}\mbox{ d}r\mbox{ d}t\\
 & =\intop_{0}^{a}\langle M_{0}u(t),\left(\mathcal{L}_{\rho}^{\ast}\psi\right)(t)\rangle_{H}\e^{-2\rho t}\mbox{ d}t+\intop_{0}^{a}\left\langle \intop_{0}^{r}M_{1}u(t)\mbox{ d}t,\left(\mathcal{L}_{\rho}^{\ast}\psi\right)(r)\right\rangle _{H}\e^{-2\rho r}\mbox{ d}r\\
 & =\intop_{0}^{a}\left\langle M_{0}u(t)+\intop_{0}^{t}M_{1}u(r)\mbox{ d}r,\left(\mathcal{L}_{\rho}^{\ast}\psi\right)(t)\right\rangle _{H}\e^{-2\rho t}\mbox{ d}t.
\end{align*}
By using (ii) and \prettyref{eq:Lpsisup}, we further obtain
\begin{align*}
 & \intop_{\mathbb{R}}\langle u(t),\varphi(t)\rangle_{H}\e^{-2\rho t}\mbox{ d}t\\
 & =\intop_{0}^{a}\langle M_{0}u_{0},\left(\mathcal{L}_{\rho}^{\ast}\psi\right)(t)\rangle_{H}\e^{-2\rho t}\mbox{ d}t\\
 & =\langle\chi_{\mathbb{R}_{\geq0}}M_{0}u_{0},\mathcal{L}_{\rho}^{\ast}\psi\rangle_{\rho}\\
 & =\langle\mathcal{L}_{\rho}\chi_{\mathbb{R}_{\geq0}}M_{0}u_{0},\psi\rangle_{0}\\
 & =\frac{1}{\sqrt{2\pi}}\intop_{\mathbb{R}}\left\langle \frac{1}{\i t+\rho}M_{0}u_{0},(-\i t+\rho)\left((-\i t+\rho)M_{0}^{\ast}+M_{1}^{\ast}\right)^{-1}(\mathcal{L}_{\rho}\varphi)(t)\right\rangle _{H}\mbox{ d}t\\
 & =\frac{1}{\sqrt{2\pi}}\intop_{\mathbb{R}}\left\langle \left((\i t+\rho)M_{0}+M_{1}\right)^{-1}M_{0}u_{0},(\mathcal{L}_{\rho}\varphi)(t)\right\rangle _{H}\mbox{ d}t\\
 & =\left\langle \mathcal{L}_{\rho}^{\ast}\frac{1}{\sqrt{2\pi}}\left(\left(\i\cdot+\rho\right)M_{0}+M_{1}\right)^{-1}M_{0}u_{0},\varphi\right\rangle _{\rho}.
\end{align*}
Since $\varphi\in C_{c}^{\infty}(\mathbb{R};H)$ was chosen arbitrarily,
we obtain (iii).

(iii) $\Rightarrow$ (i): Consider the function 
\[
f(z)=(zM_{0}+M_{1})^{-1}M_{0}u_{0}=\frac{1}{z}M_{0}u_{0}+\frac{1}{z}\left(zM_{0}+M_{1}\right)^{-1}M_{1}u_{0}\quad(z\in\mathbb{C}_{\Re\geq\rho}).
\]
Then, $f$ is continuous and analytic in $\mathbb{C}_{\Re>\rho}$
and satisfies 
\[
|zf(z)|\leq\left(\|M_{0}\|+\|\mathcal{M}(z)^{-1}\|\|M_{1}\|\right)|u_{0}|.
\]
Thus, by \prettyref{lem:causality_1} we infer that 
\[
u=\frac{1}{\sqrt{2\pi}}\mathcal{L}_{\rho}^{\ast}\left(\left(\i\cdot+\rho\right)M_{0}+M_{1}\right)^{-1}M_{0}u_{0}=\frac{1}{\sqrt{2\pi}}\mathcal{L}_{\rho}^{\ast}f(\i\cdot+\rho)
\]
is supported on $\mathbb{R}_{\geq0}$. Moreover, we compute
\begin{align}
(\i t+\rho)\left(\mathcal{L}_{\rho}M_{0}\left(u-\chi_{\mathbb{R}_{\geq0}}u_{0}\right)\right)(t) & =(\i t+\rho)M_{0}\left(\mathcal{L}_{\rho}u\right)(t)-\frac{1}{\sqrt{2\pi}}M_{0}u_{0}\nonumber \\
 & =\frac{1}{\sqrt{2\pi}}\left((\i t+\rho)M_{0}\left(\left(\i t+\rho\right)M_{0}+M_{1}\right)^{-1}-1\right)M_{0}u_{0}\nonumber \\
 & =-\frac{1}{\sqrt{2\pi}}M_{1}\left(\left(\i t+\rho\right)M_{0}+M_{1}\right)^{-1}M_{0}u_{0}\nonumber \\
 & =-M_{1}\left(\mathcal{L}_{\rho}u\right)(t)\label{eq:derivative}
\end{align}
for each $t\in\mathbb{R}$. Hence, since $M_{1}\mathcal{L}_{\rho}u\in L_{2}(\mathbb{R};H)$
we get $M_{0}(u-\chi_{\mathbb{R}_{\geq0}}u_{0})\in H_{\rho}^{1}(\mathbb{R};H).$
Therefore, $M_{0}(u-\chi_{\mathbb{R}_{\geq0}}u_{0})$ is continuous
by \prettyref{prop:Sobolev} and thus, 
\[
\lim_{t\to0+}M_{0}\left(u(t)-u_{0}\right)=0,
\]
since $u$ is supported on $\mathbb{R}_{\geq0}$. In particular, by
the fundamental theorem of calculus, we have
\[
\lim_{t\to0+}\frac{1}{t}\intop_{0}^{t}|M_{0}(u(s)-u_{0})|=0.
\]
For showing the remaining claim, let $\varphi\in C_{c}^{\infty}(\mathbb{R}_{>0}).$
From \prettyref{eq:derivative} we read off with the help of \prettyref{prop:FLtransform}
\[
\partial_{0,\rho}M_{0}(u-\chi_{\mathbb{R}_{>0}}u_{0})+M_{1}u=0
\]
and consequently, 
\begin{align*}
 & -\intop_{0}^{\infty}\langle M_{0}u(t),\varphi'(t)\rangle_{H}\mbox{ d}t+\intop_{0}^{\infty}\langle M_{1}u(t),\varphi(t)\rangle_{H}\mbox{ d}t\\
 & =-\intop_{0}^{\infty}\langle M_{0}\left(u(t)-u_{0}\right),\varphi'(t)\rangle_{H}\mbox{ d}t+\intop_{0}^{\infty}\langle M_{1}u(t),\varphi(t)\rangle_{H}\mbox{ d}t\\
 & =-\langle M_{0}(u-\chi_{\mathbb{R}_{\geq0}}u_{0}),\varphi'\e^{2\rho(\cdot)}\rangle_{\rho}+\langle M_{1}u,\varphi\e^{2\rho(\cdot)}\rangle_{\rho}\\
 & =\langle M_{0}(u-\chi_{\mathbb{R}_{\geq0}}u_{0}),-\left(\varphi\e^{2\rho(\cdot)}\right)'+2\rho\varphi\e^{2\rho(\cdot)}\rangle_{\rho}+\langle M_{1}u,\varphi\e^{2\rho(\cdot)}\rangle_{\rho}\\
 & =\langle M_{0}(u-\chi_{\mathbb{R}_{\geq0}}u_{0}),\partial_{0,\rho}^{\ast}\left(\varphi\e^{2\rho(\cdot)}\right)\rangle_{\rho}+\langle M_{1}u,\varphi\e^{2\rho(\cdot)}\rangle_{\rho}\\
 & =0.\tag*{\qedhere}
\end{align*}
\end{proof}
In the following, we shall address the asymptotic properties of the
solution $u$, if considered as a solution to \prettyref{eq:ivp_1}
or to \prettyref{eq:mIVP}. For this, we associate a spectrum to \prettyref{eq:ivp_1}
and \prettyref{eq:mIVP}. The next section proves equality of the
spectra to be introduced.

\section{$\sigma_{\textnormal{IV}}\left(\mathcal{M}\right)=\sigma\left(\mathcal{M}\right).$\label{sec:spectra}}

In this section, we will assume throughout that $\mathcal{M}$ is
a regular linear operator pencil associated with $(M_{0},M_{1})$
in some Hilbert space with $R(M_{0})\subseteq H$ closed. We consider
the following two spectra
\begin{align*}
\sigma_{\mathrm{IV}}(\mathcal{M}) & =\sigma\left(-\left(\iota_{R(M_{0})}^{\ast}M_{0}\iota_{\mathrm{IV}}\right)^{-1}\left(\iota_{R(M_{0})}^{\ast}M_{1}\iota_{\mathrm{IV}}\right)\right),\\
\sigma(\mathcal{M}) & =\{z\in\mathbb{C}\,;\,0\in\sigma(zM_{0}+M_{1})\}.
\end{align*}
We note that the first spectrum is strongly related to the asymptotics
of the initial value problem treated in \prettyref{sec:firstIV},
while the second spectrum determines the asymptotic behaviour of the
initial value problem studied in \prettyref{sec:secondIV} (for details
see \prettyref{sec:Asymptotics}). Before we address the spectra just
introduced, we elaborate a bit more on the connection of $\textnormal{IV}$
and $N(M_{0})^{\bot},$ see also \prettyref{rem:IVcongNbot}. We reformulate
$\textnormal{IV}$ first:
\begin{lem}
\label{lem:IVchar}Let $u\in H$. Then $u\in\textnormal{IV}$ if,
and only if, 
\begin{equation}
\iota_{N(M_{0})}^{*}u=-\left(\iota_{R(M_{0})^{\bot}}^{\ast}M_{1}\iota_{N(M_{0})}\right)^{-1}\iota_{R(M_{0})^{\bot}}^{\ast}M_{1}\iota_{N(M_{0})^{\bot}}\iota_{N(M_{0})^{\bot}}^{*}u.\label{eq:IVchar}
\end{equation}
\end{lem}

\begin{proof}
With the unitary operators $U_{0}$ and $U_{1}$ as introduced in
\prettyref{prop:regPen}, we have
\begin{align*}
U_{1}M_{1}U_{0}^{*} & =\left(\begin{array}{cc}
\iota_{R(M_{0})}^{\ast}M_{1}\iota_{N(M_{0})^{\bot}} & \iota_{R(M_{0})}^{\ast}M_{1}\iota_{N(M_{0})}\\
\iota_{R(M_{0})^{\bot}}^{\ast}M_{1}\iota_{N(M_{0})^{\bot}} & \iota_{R(M_{0})^{\bot}}^{\ast}M_{1}\iota_{N(M_{0})}
\end{array}\right),\\
U_{1}M_{0}U_{0}^{*} & =\left(\begin{array}{cc}
\iota_{R(M_{0})}^{\ast}M_{0}\iota_{N(M_{0})^{\bot}} & 0\\
0 & 0
\end{array}\right).
\end{align*}
Thus, for $u\in H$ with $(x,y)\coloneqq U_{0}u$ we obtain
\begin{align*}
u\in\textnormal{IV} & \iff\exists v\in H\colon M_{1}u=M_{0}v\\
 & \iff\exists v\in H\colon U_{1}M_{1}U_{0}^{*}U_{0}u=U_{1}M_{0}U_{0}^{*}U_{0}v\\
 & \iff\exists v\in H\colon\left(\begin{array}{cc}
\iota_{R(M_{0})}^{\ast}M_{1}\iota_{N(M_{0})^{\bot}} & \iota_{R(M_{0})}^{\ast}M_{1}\iota_{N(M_{0})}\\
\iota_{R(M_{0})^{\bot}}^{\ast}M_{1}\iota_{N(M_{0})^{\bot}} & \iota_{R(M_{0})^{\bot}}^{\ast}M_{1}\iota_{N(M_{0})}
\end{array}\right)\left(\begin{array}{c}
x\\
y
\end{array}\right)\\
 & \quad\quad\quad\quad\quad=\left(\begin{array}{c}
\iota_{R(M_{0})}^{\ast}M_{0}\iota_{N(M_{0})^{\bot}}\iota_{N(M_{0})^{\bot}}^{\ast}v\\
0
\end{array}\right)\\
 & \iff\exists v\in H\colon\left(\begin{array}{c}
\iota_{R(M_{0})}^{\ast}M_{1}\iota_{N(M_{0})^{\bot}}x+\iota_{R(M_{0})}^{\ast}M_{1}\iota_{N(M_{0})}y\\
\iota_{R(M_{0})^{\bot}}^{\ast}M_{1}\iota_{N(M_{0})^{\bot}}x+\iota_{R(M_{0})^{\bot}}^{\ast}M_{1}\iota_{N(M_{0})}y
\end{array}\right)\\
 & \quad\quad\quad\quad\quad=\left(\begin{array}{c}
\iota_{R(M_{0})}^{\ast}M_{0}\iota_{N(M_{0})^{\bot}}\iota_{N(M_{0})^{\bot}}^{\ast}v\\
0
\end{array}\right)\\
 & \iff\iota_{R(M_{0})^{\bot}}^{\ast}M_{1}\iota_{N(M_{0})^{\bot}}x+\iota_{R(M_{0})^{\bot}}^{\ast}M_{1}\iota_{N(M_{0})}y=0,
\end{align*}
where in the last equivalence we have used the invertibility of $\iota_{R(M_{0})}^{\ast}M_{0}\iota_{N(M_{0})^{\bot}}.$
Since, $\mathcal{M}$ is regular and $R(M_{0})\subseteq H$ is closed,
we obtain the assertion by \prettyref{prop:regPen}. Indeed, an application
of \prettyref{prop:regPen} yields that $\iota_{R(M_{0})^{\bot}}^{\ast}M_{1}\iota_{N(M_{0})}$
is an isomorphism. Thus, for $U_{0}u=(\iota_{N(M_{0})^{\bot}}^{\ast}u,\iota_{N(M_{0})}^{\ast}u)=(x,y)$
we have
\[
u\in\mathrm{IV}\iff y=-\left(\iota_{R(M_{0})^{\bot}}^{\ast}M_{1}\iota_{N(M_{0})}\right)^{-1}\iota_{R(M_{0})^{\bot}}^{\ast}M_{1}\iota_{N(M_{0})^{\bot}}x.\tag*{\qedhere}
\]
 
\end{proof}
\begin{prop}
(a) We have $N(M_{0})^{\bot}\subseteq\textnormal{IV}$ if, and only
if, $\iota_{R(M_{0})^{\bot}}^{\ast}M_{1}\iota_{N(M_{0})^{\bot}}=0$.
In either case we have $N(M_{0})^{\bot}=\mathrm{IV}$.

(b) We have $\textnormal{IV}\cap N(M_{0})^{\bot}=\{0\}$ if, and only
if, $\iota_{R(M_{0})^{\bot}}^{\ast}M_{1}\iota_{N(M_{0})^{\bot}}$
is injective.
\end{prop}

\begin{proof}
We shall prove (a) first. Using \prettyref{lem:IVchar} we observe
\begin{align*}
N(M_{0})^{\bot}\subseteq\mathrm{IV} & \Longleftrightarrow\forall x\in N(M_{0})^{\bot}:\,\left(\iota_{R(M_{0})^{\bot}}^{\ast}M_{1}\iota_{N(M_{0})}\right)^{-1}\iota_{R(M_{0})^{\bot}}^{\ast}M_{1}\iota_{N(M_{0})^{\bot}}x=0\\
 & \Longleftrightarrow\iota_{R(M_{0})^{\bot}}^{\ast}M_{1}\iota_{N(M_{0})^{\bot}}=0,
\end{align*}
which shows the asserted equivalence. Moreover, if $N(M_{0})^{\bot}\subseteq\mathrm{IV},$
we infer
\[
\left(\iota_{R(M_{0})}^{\ast}M_{0}\iota_{\textnormal{IV}}\right)^{-1}\iota_{R(M_{0})}^{\ast}M_{0}\iota_{N(M_{0})^{\bot}}=1_{N(M_{0})^{\bot}},
\]
where we used that by \prettyref{prop:generator} $\iota_{R(M_{0})}^{\ast}M_{0}\iota_{\textnormal{IV}}$
is an isomorphism. As 
\[
\left(\iota_{R(M_{0})}^{\ast}M_{0}\iota_{\textnormal{IV}}\right)^{-1}\iota_{R(M_{0})}^{\ast}M_{0}\iota_{N(M_{0})^{\bot}}
\]
 is an isomorphism mapping $N(M_{0})^{\bot}$ onto $\mathrm{IV}$
(see \prettyref{rem:IVcongNbot}), we deduce $N(M_{0})^{\bot}=\textnormal{IV}$. 

For the proof of (b), we observe that by \prettyref{lem:IVchar}
\begin{align*}
x\in\mathrm{IV}\cap N(M_{0})^{\bot} & \Longleftrightarrow x\in N(M_{0})^{\bot}\wedge\left(\iota_{R(M_{0})^{\bot}}^{\ast}M_{1}\iota_{N(M_{0})}\right)^{-1}\iota_{R(M_{0})^{\bot}}^{\ast}M_{1}\iota_{N(M_{0})^{\bot}}x=0\\
 & \Longleftrightarrow x\in\ker(\iota_{R(M_{0})^{\bot}}^{\ast}M_{1}\iota_{N(M_{0})^{\bot}}).
\end{align*}
Hence, $\mathrm{IV}\cap N(M_{0})^{\bot}=\{0\}$ if and only if $\iota_{R(M_{0})^{\bot}}^{\ast}M_{1}\iota_{N(M_{0})^{\bot}}$
is one-to-one.
\end{proof}
\begin{thm}
We have $\sigma_{\textnormal{IV}}(\mathcal{M})=\sigma(\mathcal{M}).$
\end{thm}

\begin{proof}
Recall from \prettyref{prop:regPen} that 
\[
\sigma(\mathcal{M})=\sigma(-\tilde{M}_{0}^{-1}\tilde{M}_{1}),
\]
where 
\[
\tilde{M}_{0}=\iota_{R(M_{0})}^{\ast}M_{0}\iota_{N(M_{0})^{\bot}}
\]
and 
\[
\tilde{M}_{1}=\iota_{R(M_{0})}^{*}M_{1}\iota_{N(M_{0})^{\bot}}-\iota_{R(M_{0})}^{\ast}M_{1}\iota_{N(M_{0})}\left(\iota_{R(M_{0})^{\bot}}^{\ast}M_{1}\iota_{N(M_{0})}\right)^{-1}\iota_{R(M_{0})^{\bot}}M_{1}\iota_{N(M_{0})^{\bot}}.
\]
By definition, we have
\[
\sigma_{\textnormal{IV}}(\mathcal{M})=\sigma\left(-\left(\iota_{R(M_{0})}^{\ast}M_{0}\iota_{\textnormal{IV}}\right)^{-1}\iota_{R(M_{0})}^{\ast}M_{1}\iota_{\textnormal{IV}}\right).
\]
For $z\in\mathbb{C}$ we show $z\tilde{M}_{0}+\tilde{M}_{1}$ is continuously
invertible if, and only if, 
\[
z\iota_{R(M_{0})}^{\ast}M_{0}\iota_{\textnormal{IV}}+\iota_{R(M_{0})}^{\ast}M_{1}\iota_{\textnormal{IV}}\text{ is continuously invertible.}
\]
For this, let $u\in R(M_{0})$ and $v\in\textnormal{IV}$ be such
that 
\[
z\iota_{R(M_{0})}^{\ast}M_{0}\iota_{\textnormal{IV}}v+\iota_{R(M_{0})}^{\ast}M_{1}\iota_{\textnormal{IV}}v=u.
\]
The latter is the same as saying 
\begin{align*}
u & =z\iota_{R(M_{0})}^{\ast}M_{0}\left(\begin{array}{cc}
\iota_{N(M_{0})^{\bot}} & \iota_{N(M_{0})}\end{array}\right)\left(\begin{array}{c}
\iota_{N(M_{0})^{\bot}}^{*}\\
\iota_{N(M_{0})}^{*}
\end{array}\right)\iota_{\textnormal{IV}}v\\
 & \quad+\iota_{R(M_{0})}^{\ast}M_{1}\left(\begin{array}{cc}
\iota_{N(M_{0})^{\bot}} & \iota_{N(M_{0})}\end{array}\right)\left(\begin{array}{c}
\iota_{N(M_{0})^{\bot}}^{*}\\
\iota_{N(M_{0})}^{*}
\end{array}\right)\iota_{\textnormal{IV}}v
\end{align*}
By \prettyref{lem:IVchar}, we deduce that the latter can be written
as
\begin{align*}
u & =\left(z\iota_{R(M_{0})}^{\ast}M_{0}\begin{array}{cc}
\iota_{N(M_{0})^{\bot}} & 0\end{array}\right)\left(\begin{array}{c}
\iota_{N(M_{0})^{\bot}}^{*}v\\
-\left(\iota_{R(M_{0})^{\bot}}^{\ast}M_{1}\iota_{N(M_{0})}\right)^{-1}\iota_{R(M_{0})^{\bot}}^{\ast}M_{1}\iota_{N(M_{0})^{\bot}}\iota_{N(M_{0})^{\bot}}^{*}v
\end{array}\right)\\
 & \quad+\left(\iota_{R(M_{0})}^{\ast}M_{1}\begin{array}{cc}
\iota_{N(M_{0})^{\bot}} & \iota_{R(M_{0})}^{\ast}M_{1}\iota_{N(M_{0})}\end{array}\right)\times\\
 & \quad\times\left(\begin{array}{c}
\iota_{N(M_{0})^{\bot}}^{*}v\\
-\left(\iota_{R(M_{0})^{\bot}}^{\ast}M_{1}\iota_{N(M_{0})}\right)^{-1}\iota_{R(M_{0})^{\bot}}^{\ast}M_{1}\iota_{N(M_{0})^{\bot}}\iota_{N(M_{0})^{\bot}}^{*}v
\end{array}\right)\\
 & =z\tilde{M}_{0}\iota_{N(M_{0})^{\bot}}^{\ast}v+\tilde{M}_{1}\iota_{N(M_{0})^{\bot}}^{\ast}v.
\end{align*}
Hence, for $u\in R(M_{0})$ and $v\in\textnormal{IV}$ the equation
\begin{equation}
z\iota_{R(M_{0})}^{\ast}M_{0}\iota_{\textnormal{IV}}v+\iota_{R(M_{0})}^{\ast}M_{1}\iota_{\textnormal{IV}}v=u\label{eq:IVeq}
\end{equation}
implies
\begin{equation}
z\tilde{M}_{0}\iota_{N(M_{0})^{\bot}}^{\ast}v+\tilde{M}_{1}\iota_{N(M_{0})^{\bot}}^{\ast}v=u.\label{eq:peneq}
\end{equation}
On the other hand, following the argument in reverse direction, we
get that if $x\in N(M_{0})^{\bot}$ solves \prettyref{eq:peneq},
then 
\[
v\coloneqq\left(\begin{array}{cc}
\iota_{N(M_{0})^{\bot}} & \iota_{N(M_{0})}\end{array}\right)\left(\begin{array}{c}
\iota_{N(M_{0})^{\bot}}^{*}x\\
-\left(\iota_{R(M_{0})^{\bot}}^{\ast}M_{1}\iota_{N(M_{0})}\right)^{-1}\iota_{R(M_{0})^{\bot}}^{\ast}M_{1}\iota_{N(M_{0})^{\bot}}\iota_{N(M_{0})^{\bot}}^{*}x
\end{array}\right)
\]
solves \prettyref{eq:IVeq}. Note that $v\in\textnormal{IV}$, by
\prettyref{lem:IVchar}. The assertion follows. 
\end{proof}
In the next section we discuss the asymptotic properties of the solutions
of both initial value problems \prettyref{eq:ivp_1} and \prettyref{eq:mIVP}. 

\section{\label{sec:Asymptotics}Asymptotic Properties}

In the previous section, we have shown that the spectra connected
to strong and mild solutions coincide. Hence, one might argue that
the asymptotic properties (i.e., exponential stability or instability)
of solutions to \prettyref{eq:ivp_1} and \prettyref{eq:mIVP} are
the same. For this, we note here that for mild solutions it does not
make sense to talk about pointwise properties since mild solutions
are in a certain $L_{2}$-space, only. We shall however introduce
the correct concepts anticipating this problem. It will turn out that
the exponential weight in the $L_{2}$-spaces considered in this exposition
can be used to define asymptotically stable parts. 

We assume throughout this section that $\mathcal{M}$ is a regular
linear operator pencil associated with $(M_{0},M_{1})$ and that the
range of $M_{0}$ is a closed subspace of $H$. Moreover, we consider
the linear pencil $\mathcal{N}$ associated with $(M_{0},-M_{1}).$
Since $\mathcal{M}$ is regular, we have by \prettyref{prop:regPen}
that $\sigma(\mathcal{M})$ is compact and that $\|\mathcal{M}(z)^{-1}\|$
is uniformly bounded outside a ball $B(0,\rho)\subseteq\mathbb{C}$
for some $\rho>0$ large enough. Since 
\[
\mathcal{N}(z)=zM_{0}-M_{1}=-\left(-zM_{0}+M_{1}\right)=-\mathcal{M}(-z),\quad(z\in\mathbb{C})
\]
we infer that also $\mathcal{N}$ is regular and $\|\mathcal{N}(z)^{-1}\|$
is uniformly bounded outside the same ball $B(0,\rho).$ We first
show that there is a strong connection between the solutions of the
initial value problems induced by $\mathcal{M}$ and by $\mathcal{N}$. 

We recall from \prettyref{thm:IVP_2} that for each $u_{0}\in H$
the mild solution $u$ of \prettyref{eq:ivp_1} satisfies
\[
u=\mathcal{L}_{\rho}^{\ast}\left(t\mapsto\frac{1}{\sqrt{2\pi}}\left(\mathcal{M}(\i t+\rho)\right)^{-1}M_{0}u_{0}\right).
\]
Moreover, we have that $\mathbb{R}_{\geq0}\ni t\mapsto\left(M_{0}u\right)(t)$
is a continuous function. Hence, point-evaluation of $M_{0}u$ is
well-defined. This will be used in the next statement.
\begin{prop}
\label{prop:forwadr_backward} Let $\mathcal{N}$ be the linear pencil
associated with $(M_{0},-M_{1})$ and $\rho>0$ large enough, such
that $\sup_{z\in\mathbb{C}\setminus B(0,\rho)}\|\mathcal{M}(z)^{-1}\|<\infty.$
Moreover, let $u_{0}\in H$ and let $u\in L_{2,\rho}(\mathbb{R};H)$
be given by 
\[
u\coloneqq\mathcal{L}_{\rho}^{\ast}\left(t\mapsto\frac{1}{\sqrt{2\pi}}\left(\mathcal{M}(\i t+\rho)\right)^{-1}M_{0}u_{0}\right).
\]
Let $T>0$ and define 
\[
w\coloneqq\mathcal{L}_{\rho}^{\ast}\left(t\mapsto\frac{1}{\sqrt{2\pi}}\left(\mathcal{N}(\i t+\rho)\right)^{-1}\left(M_{0}u\right)(T)\right).
\]
Then $w=u(T-\cdot)$ on $[0,T].$ 
\end{prop}

\begin{proof}
We consider the following function 
\[
v(t)\coloneqq\chi_{\mathbb{R}_{\geq0}}(t+T)(w(t+T)-u(-t))\quad(t\in\mathbb{R}).
\]
We note that $v\in L_{2,\rho}(\mathbb{R};H)$ since $u,w\in L_{2,\rho}(\mathbb{R};H)$
and $\spt u\subseteq\mathbb{R}_{\geq0}$ by \prettyref{thm:IVP_2}.
We claim that 
\begin{equation}
\left(\mathcal{L}_{\rho}v\right)(t)=\left(\mathcal{N}(\i t+\rho)\right)^{-1}M_{0}u_{0}.\label{eq:eq_v}
\end{equation}
If the latter is true, the assertion follows, since then $\spt v\subseteq\mathbb{R}_{\geq0}$
by \prettyref{thm:IVP_2} and thus, 
\[
w(t)-u(T-t)=v(t-T)=0\quad(t\in[0,T]\text{ a.e.}).
\]
To verify \prettyref{eq:eq_v}, we compute using that $\spt u,\spt w\subseteq\mathbb{R}_{\geq0}$,
\begin{align*}
\left(\mathcal{L}_{\rho}v\right)(t) & =\e^{(\i t+\rho)T}\left(\mathcal{\mathcal{L}_{\rho}}w(t)\right)-\frac{1}{\sqrt{2\pi}}\intop_{-T}^{0}\e^{-(\i t+\rho)s}u(-s)\text{ d}s\quad(t\in\mathbb{R})
\end{align*}
and thus, 
\[
\mathcal{N}(\i t+\rho)\left(\mathcal{L}_{\rho}v\right)(t)=\e^{(\i t+\rho)T}\frac{1}{\sqrt{2\pi}}\left(M_{0}u\right)(T)-\frac{1}{\sqrt{2\pi}}\intop_{-T}^{0}\e^{-(\i t+\rho)s}\mathcal{N}(\i t+\rho)u(-s)\text{ d}s\quad(t\in\mathbb{R}).
\]
Since $u$ satisfies \prettyref{eq:IVP_2} by \prettyref{thm:IVP_2},
we may compute 
\begin{align*}
 & \intop_{-T}^{0}\e^{-(\i t+\rho)s}\mathcal{N}(\i t+\rho)u(-s)\text{ d}s\\
 & =\intop_{-T}^{0}\e^{-(\i t+\rho)s}\left(\left(\i t+\rho\right)M_{0}u(-s)-M_{1}u(-s)\right)\text{ d}s\\
 & =\intop_{-T}^{0}\e^{-(\i t+\rho)s}\left(\left(\i t+\rho\right)M_{0}u_{0}-(\i t+\rho)\intop_{0}^{-s}M_{1}u(r)\text{ d}r-M_{1}u(-s)\right)\text{ d}s.
\end{align*}
Since 
\begin{align*}
\intop_{-T}^{0}-\e^{-(\i t+\rho)s}(\i t+\rho)\intop_{0}^{-s}M_{1}u(r)\text{ d}r\text{ d}s & =\intop_{0}^{T}M_{1}u(r)\intop_{-T}^{-r}-\e^{-(\i t+\rho)s}(\i t+\rho)\text{ d}s\text{ d}r\\
 & =\intop_{0}^{T}M_{1}u(r)\left(\e^{(\i t+\rho)r}-\e^{(\i t+\rho)T}\right)\text{ d}r\\
 & =\intop_{-T}^{0}M_{1}u(-s)\left(\e^{-(\i t+\rho)s}-\e^{(\i t+\rho)T}\right)\text{ d}s
\end{align*}
we infer that 
\begin{align*}
 & \intop_{-T}^{0}\e^{-(\i t+\rho)s}\mathcal{N}(\i t+\rho)u(-s)\text{ d}s\\
 & =-(1-\e^{(\i t+\rho)T})M_{0}u_{0}-\intop_{-T}^{0}\e^{-(\i t+\rho)s}M_{1}u(-s)\text{ d}s+\intop_{-T}^{0}M_{1}u(-s)\left(\e^{-(\i t+\rho)s}-\e^{(\i t+\rho)T}\right)\text{ d}s\\
 & =\e^{(\i t+\rho)T}\left(M_{0}u_{0}-\intop_{-T}^{0}M_{1}u(-s)\text{ d}s\right)-M_{0}u_{0}\\
 & =\e^{(\i t+\rho)T}\left(M_{0}u\right)(T)-M_{0}u_{0},
\end{align*}
where we have again used \prettyref{thm:IVP_2}. Summarising, we obtain
\begin{align*}
\mathcal{N}(\i t+\rho)\left(\mathcal{L}_{\rho}v\right)(t) & =\e^{(\i t+\rho)T}\frac{1}{\sqrt{2\pi}}\left(M_{0}u\right)(T)-\frac{1}{\sqrt{2\pi}}\intop_{-T}^{0}\e^{-(\i t+\rho)s}\mathcal{N}(\i t+\rho)u(-s)\text{ d}s\\
 & =\frac{1}{\sqrt{2\pi}}M_{0}u_{0},
\end{align*}
which proves \prettyref{eq:eq_v}.
\end{proof}
We introduce the following two solution operators: Let $\mathcal{N}$
the linear pencil associated with $(M_{0},-M_{1})$ and let $\rho>0$
such that $\sup_{z\in\mathbb{C}\setminus B(0,\rho)}\|\mathcal{M}(z)^{-1}\|<\infty.$
Then we define 
\begin{align*}
S(\mathcal{M}):H & \to L_{2,\rho}(\mathbb{R};H)\\
u_{0} & \mapsto\mathcal{L}_{\rho}^{\ast}\left(t\mapsto\frac{1}{\sqrt{2\pi}}\left(\mathcal{M}(\i t+\rho)\right)^{-1}M_{0}u_{0}\right)
\end{align*}
and 
\begin{align*}
S(\mathcal{N}):H & \to L_{2,\rho}(\mathbb{R};H)\\
w_{0} & \mapsto\mathcal{L}_{\rho}^{\ast}\left(t\mapsto\frac{1}{\sqrt{2\pi}}\left(\mathcal{N}(\i t+\rho)\right)^{-1}M_{0}w_{0}\right).
\end{align*}
Using this notation, we realise that \prettyref{prop:forwadr_backward}
states that for all $u_{0}\in H$ and $T>0$
\[
S(\mathcal{N})\left(\left(S(\mathcal{M})u_{0}\right)(T)\right)=\left(S(\mathcal{M})u_{0}\right)(T-\cdot)
\]
on $[0,T].$\\
\begin{defn}
Let $S\subseteq\mathrm{IV}$ be a closed subspace. Then we call $S$
\emph{invariant under $\mathcal{M}$}, if for all $u_{0}\in S$, we
have that $u(t)\in S$ for all $t\in\mathbb{R}_{>0}$, where $u$
is the (strong) solution of \prettyref{eq:ivp_1}.
\end{defn}

\begin{rem}
Let $\rho>0$ be such that $\sup_{z\in\mathbb{C}\setminus B(0,\rho)}\|\mathcal{M}(z)^{-1}\|<\infty$.
By \prettyref{prop:LaplaceRep}, we note that the invariance of the
closed subspace $S\subseteq\mathrm{IV}$ is equivalent to the condition
\[
S(\mathcal{M})[S]\subseteq L_{2,\mu}(\mathbb{R};S),
\]
for all $\mu\geq\rho$, which in turn is equivalent to 
\[
\left(\left(\mathcal{M}(\i t+\mu)\right)^{-1}M_{0}\right)[S]\subseteq S
\]
for every $t\in\mathbb{R},\mu\geq\rho.$ By the identity theorem,
the latter is equivalent to 
\[
\left(\left(\mathcal{M}(z)\right)^{-1}M_{0}\right)[S]\subseteq S
\]
for all $z\in\mathbb{C}\setminus B(0,\rho)$. 
\end{rem}

\begin{defn}
We say that $\mathcal{M}$ admits a \emph{strong exponential dichotomy},
if there exist closed subspaces $S,T\subseteq\textnormal{IV}$ with
the following properties: $S\dotplus T=\textnormal{IV},$ $S$ and
$T$ are invariant under $\mathcal{M}$ and there is $\rho>0,C\geq0$
such that for all $t\geq0$ 
\[
\begin{cases}
|u(t)|_{H}\leq C\e^{-\rho t}|u_{0}|_{H} & u_{0}\in S,\\
|u(t)|_{H}\geq C\e^{\rho t}|u_{0}|_{H} & u_{0}\in T,
\end{cases}
\]
where $u\coloneqq S(\mathcal{M})u_{0}.$ Moreover, we say that $\mathcal{M}$
is \emph{strongly exponentially stable}, if it admits a strong exponential
dichotomy with $S=\mathrm{IV}$ and $T=\{0\}.$ 
\end{defn}

We now reformulate the exponential growth on the subspace $T$ as
an exponential decay for the pencil $\mathcal{N}.$ We note that the
$\mathrm{IV}$-spaces for $\mathcal{M}$ and $\mathcal{N}$ coincide,
which is a consequence of $\mathrm{IV}=M_{1}^{-1}[R(M_{0})]=-M_{1}^{-1}[R(M_{0})]$,
by the linearity of $M_{1}$ and $M_{0}$.
\begin{lem}
Let $T\subseteq\mathrm{IV}$ be a closed subspace and invariant under
$\mathcal{M}$, $\rho>0$. Then the following statements are equivalent:

\begin{enumerate}[(i)]

\item There exists $C>0$ such that 
\[
|\left(S(\mathcal{M})u_{0}\right)(t)|_{H}\geq C\e^{\rho t}|u_{0}|_{H}
\]
for all $t\geq0,u_{0}\in T$. 

\item There exists $C>0$ such that 
\[
\left|\left(S(\mathcal{N})w_{0}\right)(t)\right|_{H}\leq C\e^{-\rho t}|w_{0}|_{H}
\]
for all $t\geq0,w_{0}\in T.$

\end{enumerate}
\end{lem}

\begin{proof}
(i) $\Rightarrow$ (ii): Let $w_{0}\in T$ and set $w\coloneqq S(\mathcal{N})w_{0}$.
We first prove that $w(t)\in T$ for each $t\geq0.$ Indeed, since
\[
\left(\left(\mathcal{M}(z)\right)^{-1}M_{0}\right)[T]\subseteq T,
\]
we infer that also
\[
\left(\left(\mathcal{N}(z)\right)^{-1}M_{0}\right)[T]=\left(-\left(\mathcal{M}(-z)\right)^{-1}M_{0}\right)[T]\subseteq T
\]
for all $z\in\mathbb{C}\setminus B(0,\rho')$ for $\rho'$ large enough.
Hence, $T$ is invariant under $\mathcal{N}$ and thus $w(t)\in T$
for each $t\geq0.$ We fix $t>0$ and set $u\coloneqq S(\mathcal{M})\left(w(t)\right).$
Since $w(t)\in T$ we obtain 
\begin{equation}
|u(s)|_{H}\geq C\e^{\rho s}|w(t)|_{H}\label{eq:unstable}
\end{equation}
for each $s\geq0.$ Using now \prettyref{prop:forwadr_backward} (and
interchanging the roles of $\mathcal{M}$ and $\mathcal{N}$), we
derive 
\[
u=w(t-\cdot)
\]
on $[0,t]$ and hence, in particular 
\[
u(t)=w(0)=w_{0}.
\]
Thus, \prettyref{eq:unstable} gives 
\[
|w(t)|_{H}\leq C^{-1}\e^{-\rho t}|w_{0}|_{H}
\]
which shows the claim. \\
(ii) $\Rightarrow$ (i): The claim can be shown by following the same
argumentation as above.
\end{proof}
\begin{lem}
\label{lem:decay_L_2}Let $S\subseteq\mathrm{IV}$ be a closed subspace.
Then the following conditions are equivalent

\begin{enumerate}[(i)]

\item $S$ is invariant under $\mathcal{M}$ and there exists $C>0,\rho>0$
such that 
\[
|\left(S(\mathcal{M})u_{0}\right)(t)|_{H}\leq C\e^{-\rho t}|u_{0}|_{H}
\]
for all $u_{0}\in S,t\geq0$. 

\item There exists $\rho'>0$ such that 
\[
S(\mathcal{M})u_{0}\in L_{2,-\rho'}(\mathbb{R};S)
\]
for all $u_{0}\in S.$ 

\end{enumerate}
\end{lem}

\begin{proof}
(i) $\Rightarrow$ (ii): This is clear for each $0<\rho'<\rho$.\\
(ii) $\Rightarrow$ (i): Let $u_{0}\in S$. As $u\coloneqq S(\mathcal{M})u_{0}$
is continuous, we infer $\left(S(\mathcal{M})u_{0}\right)(t)\in S$
for every $t\geq0.$ Let now 
\[
\varphi(t)\coloneqq\begin{cases}
1-t & \text{ if }0\leq t\leq1,\\
0 & \text{ else}.
\end{cases}
\]
Then we compute for each $t\geq1$ using the representation of $u(t)$
from \prettyref{cor:IVsemLapl} 
\begin{align*}
|u(t)| & =\left|\intop_{0}^{t}(u'(s)-\varphi'(s)u_{0})\text{ d}s\right|\\
 & \leq\left(\intop_{0}^{t}|u'(s)-\varphi'(s)u_{0}|^{2}\e^{2\rho's}\text{ d}s\right)^{\frac{1}{2}}\frac{1}{\sqrt{2\rho'}}\e^{-\rho't}\\
 & \leq\frac{1}{\sqrt{2\rho'}}\e^{-\rho't}\left(\left|\left(\iota_{R(M_{0})}^{\ast}M_{0}\iota_{\mathrm{IV}}\right)^{-1}\left(\iota_{R(M_{0})}^{\ast}M_{1}\iota_{\mathrm{IV}}\right)u\right|_{L_{2,-\rho'}}+|u_{0}|_{H}\e^{2\rho'}\right)\\
 & \leq C\e^{-\rho't}\left(|u|_{L_{2,-\rho'}}+|u_{0}|_{H}\right).
\end{align*}
Since $S(\mathcal{M}):S\to L_{2,-\rho'}(\mathbb{R};S)$ is closed,
we infer its boundedness by the closed graph theorem. Hence, for some
$C'\geq0$ 
\[
|u(t)|\leq C'\e^{-\rho't}|u_{0}|_{H}\quad(t\geq1).
\]
By \prettyref{cor:IVsemLapl} again, we deduce for some $C''\geq0$
\[
|u(t)|\leq C''\e^{-\rho't}|u_{0}|_{H}\quad(t\geq0).
\]
 Thus, the assertion follows with $\rho=\rho'$. 
\end{proof}
With the latter two lemmas at hand, we immediately obtain the following
proposition.
\begin{prop}
$\mathcal{M}$ admits a strong exponential dichotomy if and only if
there are two closed subspaces $S,T\subseteq\mathrm{IV}$ with $S\dotplus T=\mathrm{IV}$
and $\rho>0$ such that 
\[
\begin{cases}
S(\mathcal{M})u_{0}\in L_{2,-\rho}(\mathbb{R};S), & u_{0}\in S,\\
S(\mathcal{N})u_{0}\in L_{2,-\rho}(\mathbb{R};T) & u_{0}\in T,
\end{cases}
\]
where $\mathcal{N}$ denotes the pencil associated with $(M_{0},-M_{1}).$
\end{prop}

Since this characterisation of a strong exponential dichotomy does
not require any continuity conditions on the solutions, we can use
this notion to define exponential dichotomies for our second initial
value problem \prettyref{eq:mIVP}. For this, observe that by \prettyref{thm:IVP_2},
we obtain that the unique solution $u$ solving \prettyref{eq:mIVP}
for $u_{0}\in N(M_{0})$ is the zero function. Thus, for the description
of asymptotic behaviour, we can dispense with the space $N(M_{0})$.
\begin{defn}
We say that $\mathcal{M}$ admits a \emph{mild exponential dichotomy
}if there exist closed subspaces $S,T\subseteq N(M_{0})^{\bot}$ satisfying
the following properties: $S\dotplus T=N(M_{0})^{\bot}$ and there
exists $\rho>0$ such that
\[
\begin{cases}
\iota_{N(M_{0})^{\bot}}^{\ast}S(\mathcal{M})u_{0}\in L_{2,-\rho}(\mathbb{R};S), & u_{0}\in S,\\
\iota_{N(M_{0})^{\bot}}^{\ast}S(\mathcal{N})u_{0}\in L_{2,-\rho}(\mathbb{R};T), & u_{0}\in T.
\end{cases}
\]
We say that $\mathcal{M}$ is \emph{mildly exponentially stable,}
if it admits a mild exponential dichotomy with $S=N(M_{0})^{\bot}$
and $T=\{0\}$.
\end{defn}

The main goal is now to provide a characterisation for exponential
dichotomy in terms of the spectrum $\sigma_{\mathrm{IV}}(\mathcal{M})=\sigma(\mathcal{M})$.
As a prerequisite, we start to study the simple case $M_{0}=1;$ moreover,
we shall characterise mild and strong exponential stability. Note
that we recover the stability theorem in \cite[Theorem 3-1.1]{Dai1989}
for the finite-dimensional case.
\begin{lem}
\label{lem:GP}Let $M\in L(H)$, $\rho_{1}\in\mathbb{R}$, and let
$\tilde{\mathcal{M}}$ be the pencil associated with $(1,M)$. Then
the following conditions are equivalent:

\begin{enumerate}[(i)]

\item For all $v\in H$ we have that $S(\tilde{\mathcal{M}})v\in\bigcap_{\rho>\rho_{1}}L_{2,\rho}(\mathbb{R};H)$. 

\item$\sigma(-M)\subseteq\mathbb{C}_{\Re\leq\rho_{1}}$.

\end{enumerate}
\end{lem}

\begin{proof}
We show that (i) implies (ii), first. Since $M$ is a bounded operator,
its spectrum is compact. Let $\rho_{2}>\rho_{1}$ with $\sigma(-M)\subseteq\mathbb{C}_{\Re<\rho_{2}}.$
Assume by contradiction that $\sigma(-M)\cap\mathbb{C}_{\Re>\rho_{1}}\neq\emptyset$.
The mapping
\[
\varrho(-M)\cap\mathbb{C}_{\Re>\rho_{1}}\ni z\mapsto(z+M)^{-1}
\]
is holomorphic. By assumption we find a convergent sequence $(z_{n})_{n}$
in $\mathbb{C}_{\Re>\rho_{1}}$ such that, for every $n\in\mathbb{N}$,
$z_{n}$ lies in the component of $\mathbb{C}_{\Re>\rho_{2}}$ in
$\varrho(-M)$ and 
\[
\|\left(z_{n}+M\right)^{-1}\|\to\infty\quad(n\to\infty).
\]
 By the uniform boundedness principle, there exists $v\in H$ such
that
\[
\left|\left(z_{n}+M\right)^{-1}v\right|_{H}\to\infty\quad(n\to\infty)
\]
Since $S(\tilde{\mathcal{M}})v\in\bigcap_{\rho>\rho_{1}}L_{2,\rho}(\mathbb{R};H)$
and $\spt S(\tilde{\mathcal{M}})v\subseteq\mathbb{R}_{\geq0}$ by
\prettyref{thm:IVP_2}, we deduce that the mapping 
\[
f\colon\mathbb{C}_{\Re>\rho_{1}}\ni z\mapsto\left(\mathcal{L}_{\Re z}S(\tilde{\mathcal{M}})v\right)(\Im z)=\frac{1}{\sqrt{2\pi}}\intop_{0}^{\infty}\e^{-zt}\left(S(\tilde{\mathcal{M}})v\right)(t)\text{ d}t
\]
is holomorphic. Moreover, by (i), we get
\[
f(z)=\frac{1}{\sqrt{2\pi}}(z+M)^{-1}v
\]
for all $z\in\mathbb{C}_{\Re>\rho_{2}}$. By the identity theorem,
we infer
\[
f(z_{n})=\frac{1}{\sqrt{2\pi}}(z_{n}+M)^{-1}v
\]
for all $n\in\mathbb{N}$. Since, $f$ is bounded on compact subsets
of $\mathbb{C}_{\Re>\rho_{1}}$, we obtain
\[
\infty>\sup_{n}\left|f(z_{n})\right|_{H}=\sup_{n}\left|\frac{1}{\sqrt{2\pi}}(z_{n}+M)^{-1}v\right|_{H}=\infty,
\]
a contradiction.

Next, we prove that (ii) implies (i). Let $v\in H$ and set $u\coloneqq S(\tilde{\mathcal{M}})v$.
By \prettyref{thm:IVP_2} we have 
\[
\left(\mathcal{L}_{\rho}u\right)(t)=\frac{1}{\sqrt{2\pi}}((\i t+\rho)+M)^{-1}v\quad(t\in\mathbb{R})
\]
for each $\rho>\max\{s_{0}(\tilde{\mathcal{M}}),0\}.$ By assumption,
the mapping 
\begin{align*}
f:\mathbb{C}_{\Re>\rho_{1}} & \to H\\
z & \mapsto\frac{1}{\sqrt{2\pi}}(z+M)^{-1}v
\end{align*}
is analytic. For $\rho'\in]\rho_{1},\rho[,$ we in particular deduce
that $f:\{z\in\mathbb{C}\,;\,\rho'\leq\Re z\leq\rho\}\to H$ is continuous
and analytic in the interior. Since for $|z|>2\|M\|$ we have 
\[
|zf(z)|\leq\frac{1}{\sqrt{2\pi}}|v|_{H}\|(1+z^{-1}M)^{-1}\|\leq\sqrt{\frac{2}{\pi}}|v|_{H}
\]
by the Neumann series, we infer 
\[
\sup_{\rho'\leq\Re z\leq\rho}|zf(z)|<\infty
\]
and thus, 
\[
u=\mathcal{L}_{\rho'}^{\ast}\left(f(\i\cdot+\rho')\right)\in L_{2,\rho'}(\mathbb{R};H)
\]
according to \prettyref{lem:independent_rho}.
\end{proof}
Due to \prettyref{cor:semigroupIVP} the next result is a direct consequence
of \prettyref{lem:GP}:
\begin{thm}
\label{thm:strongDicho}$\mathcal{M}$ admits a strong exponential
dichotomy, if, and only if, 
\[
\sigma_{\textnormal{IV}}(\mathcal{M})=\sigma\left(-\left(\iota_{R(M_{0})}^{\ast}M_{0}\iota_{\mathrm{IV}}\right)^{-1}\left(\iota_{R(M_{0})}^{\ast}M_{1}\iota_{\mathrm{IV}}\right)\right)\cap i\mathbb{R}=\emptyset.
\]
The pencil $\mathcal{M}$ is exponentially stable, if, and only if,
\[
\sigma_{\mathrm{IV}}(\mathcal{M})\subseteq\mathbb{C}_{\Re<0}.
\]
\end{thm}

\begin{proof}
If $\mathcal{M}$ admits a strong exponential dichotomy, then $\left(\iota_{R(M_{0})}^{\ast}M_{0}\iota_{\mathrm{IV}}\right)^{-1}\left(\iota_{R(M_{0})}^{\ast}M_{1}\iota_{\mathrm{IV}}\right)$
leaves the spaces $S$ and $T$ invariant. Hence,
\begin{align*}
 & \sigma\left(-\left(\iota_{R(M_{0})}^{\ast}M_{0}\iota_{\mathrm{IV}}\right)^{-1}\left(\iota_{R(M_{0})}^{\ast}M_{1}\iota_{\mathrm{IV}}\right)\right)\\
 & =\sigma\left(-\iota_{S}^{\ast}\left(\iota_{R(M_{0})}^{\ast}M_{0}\iota_{\mathrm{IV}}\right)^{-1}\left(\iota_{R(M_{0})}^{\ast}M_{1}\iota_{\mathrm{IV}}\right)\iota_{S}\right)\\ & \quad\quad\cup\sigma\left(-\iota_{T}^{\ast}\left(\iota_{R(M_{0})}^{\ast}M_{0}\iota_{\mathrm{IV}}\right)^{-1}\left(\iota_{R(M_{0})}^{\ast}M_{1}\iota_{\mathrm{IV}}\right)\iota_{T}\right)\\
 & \subseteq\mathbb{C}_{\Re<0}\cup\mathbb{C}_{\Re>0}
\end{align*}
by \prettyref{lem:decay_L_2} and \prettyref{lem:GP}. On the other
hand, if $\sigma_{\mathrm{IV}}(\mathcal{M})\cap\i\mathbb{R}=\emptyset$,
then we can choose $S\coloneqq P[\mathrm{IV}]$ and $T\coloneqq(1-P)\left[\mathrm{IV}\right],$
where $P$ denotes the Dunford projection on $\sigma_{\mathrm{IV}}(\mathcal{M})\cap\mathbb{C}_{\Re<0}$.
Indeed, we then obtain that $S$ and $T$ are left invariant and that
\begin{align*}
\sigma\left(-\iota_{S}^{\ast}\left(\iota_{R(M_{0})}^{\ast}M_{0}\iota_{\mathrm{IV}}\right)^{-1}\left(\iota_{R(M_{0})}^{\ast}M_{1}\iota_{\mathrm{IV}}\right)\iota_{S}\right) & \subseteq\mathbb{C}_{\Re<0}\\
\sigma\left(-\iota_{T}^{\ast}\left(\iota_{R(M_{0})}^{\ast}M_{0}\iota_{\mathrm{IV}}\right)^{-1}\left(\iota_{R(M_{0})}^{\ast}M_{1}\iota_{\mathrm{IV}}\right)\iota_{T}\right) & \subseteq\mathbb{C}_{\Re>0},
\end{align*}
which again yields the assertion by \prettyref{lem:decay_L_2} and
\prettyref{lem:GP}.
\end{proof}
Next, we shall address mild exponential dichotomy and stability. For
this, we discuss a preliminary observation first. 
\begin{lem}
\label{lem:refregPen}Adopt the notation from \prettyref{prop:regPen}.
Then for $z\in\mathbb{C}_{\Re>\rho}$, where $\rho>\max\{s_{0}(\mathcal{M}),0\},$
we have
\[
\iota_{N(M_{0})^{\bot}}^{\ast}\mathcal{M}(z)^{-1}M_{0}=\left(z-A\right)^{-1}\iota_{N(M_{0})^{\bot}}^{*},
\]
where $A\coloneqq-\tilde{M}_{0}^{-1}\tilde{M}_{1}$ with $\tilde{M}_{0},\tilde{M}_{1}$
iven as in \prettyref{prop:regPen}. In particular, \[\iota_{N(M_{0})^{\bot}}^{\ast}S(\mathcal{M})u_{0}=S(\tilde{\mathcal{M}})\iota_{N(M_{0})^{\bot}}^{\ast}u_{0}\]
for each $u_{0}\in H$, where $\tilde{\mathcal{M}}$ denotes the pencil
associated with $(1,-A).$ 
\end{lem}

\begin{proof}
By \prettyref{prop:regPen} we compute,
\begin{align*}
 & V_{0}U_{0}\mathcal{M}(z)^{-1}M_{0}\\
 & =V_{0}U_{0}\mathcal{M}(z)^{-1}\iota_{R(M_{0})}\iota_{R(M_{0})}^{\ast}M_{0}\iota_{N(M_{0})^{\bot}}\iota_{N(M_{0})^{\bot}}^{*}\\
 & =\left(\begin{array}{cc}
\left(z\tilde{M}_{0}+\tilde{M}_{1}\right)^{-1} & 0\\
0 & \left(\iota_{R(M_{0})^{\bot}}^{\ast}M_{1}\iota_{N(M_{0})}\right)^{-1}
\end{array}\right)V_{1}^{-1}U_{1}\iota_{R(M_{0})}\iota_{R(M_{0})}^{\ast}M_{0}\iota_{N(M_{0})^{\bot}}\iota_{N(M_{0})^{\bot}}^{*}\\
 & =\left(\begin{array}{cc}
\left(z\tilde{M}_{0}+\tilde{M}_{1}\right)^{-1} & 0\\
0 & \left(\iota_{R(M_{0})^{\bot}}^{\ast}M_{1}\iota_{N(M_{0})}\right)^{-1}
\end{array}\right)V_{1}^{-1}\left(\begin{array}{c}
\iota_{R(M_{0})}^{\ast}\\
\iota_{R(M_{0})^{\bot}}^{\ast}
\end{array}\right)\iota_{R(M_{0})}\tilde{M}_{0}\iota_{N(M_{0})^{\bot}}^{*}\\
 & =\left(\begin{array}{cc}
\left(z\tilde{M}_{0}+\tilde{M}_{1}\right)^{-1} & 0\\
0 & \left(\iota_{R(M_{0})^{\bot}}^{\ast}M_{1}\iota_{N(M_{0})}\right)^{-1}
\end{array}\right)\\
 & \quad\times\left(\begin{array}{cc}
1 & -\iota_{R(M_{0})}^{*}M_{1}\iota_{N(M_{0})}\left(\iota_{R(M_{0})^{\bot}}^{\ast}M_{1}\iota_{N(M_{0})}\right)^{-1}\\
0 & 1
\end{array}\right)\left(\begin{array}{c}
\tilde{M}_{0}\iota_{N(M_{0})^{\bot}}^{*}\\
0
\end{array}\right)\\
 & =\left(\begin{array}{cc}
\left(z\tilde{M}_{0}+\tilde{M}_{1}\right)^{-1} & 0\\
0 & \left(\iota_{R(M_{0})^{\bot}}^{\ast}M_{1}\iota_{N(M_{0})}\right)^{-1}
\end{array}\right)\left(\begin{array}{c}
\tilde{M}_{0}\iota_{N(M_{0})^{\bot}}^{*}\\
0
\end{array}\right)\\
 & =\left(\begin{array}{c}
\left(z\tilde{M}_{0}+\tilde{M}_{1}\right)^{-1}\tilde{M}_{0}\iota_{N(M_{0})^{\bot}}^{*}\\
0
\end{array}\right)\\
 & =\left(\begin{array}{c}
\left(z+\tilde{M}_{0}^{-1}\tilde{M}_{1}\right)^{-1}\iota_{N(M_{0})^{\bot}}^{*}\\
0
\end{array}\right)=\left(\begin{array}{c}
\left(z-A\right)^{-1}\iota_{N(M_{0})^{\bot}}^{*}\\
0
\end{array}\right).
\end{align*}
Hence, 
\begin{align*}
\iota_{N(M_{0})^{\bot}}^{\ast}\mathcal{M}(z)^{-1}M_{0} & =\iota_{N(M_{0})^{\bot}}^{\ast}U_{0}^{\ast}V_{0}^{-1}\left(\begin{array}{c}
\left(z-A\right)^{-1}\iota_{N(M_{0})^{\bot}}^{*}\\
0
\end{array}\right)\\
 & =\iota_{N(M_{0})^{\bot}}^{\ast}\left(\begin{array}{cc}
\iota_{N(M_{0})^{\bot}} & \iota_{N(M_{0})}\end{array}\right)\left(\begin{array}{c}
\left(z-A\right)^{-1}\iota_{N(M_{0})^{\bot}}^{*}\\
0
\end{array}\right)\\
 & =\left(z-A\right)^{-1}\iota_{N(M_{0})^{\bot}}^{*}.\tag*{\qedhere}
\end{align*}
\end{proof}
\begin{thm}
\label{thm:mild_exp_dich}$\mathcal{M}$ admits a mild exponential
dichotomy, if, and only if, 
\[
\sigma\left(\mathcal{M}\right)\cap i\mathbb{R}=\emptyset.
\]
$\mathcal{M}$ is exponentially stable, if, and only if, 
\[
\sigma\left(\mathcal{M}\right)\subseteq\mathbb{C}_{\Re<0}.
\]
\end{thm}

\begin{proof}
By \prettyref{lem:refregPen}, $\mathcal{M}$ admits a mild exponential
dichotomy, if $\tilde{\mathcal{M}}$ admits a mild exponential dichotomy
with the same subspaces $S,T$, where $\tilde{\mathcal{M}}$ denotes
the pencil associated with $(1,-A).$ The latter is equivalent to
$\sigma(A)\cap\i\mathbb{R}=\emptyset$. Indeed, if $\sigma(A)\cap\i\mathbb{R}=\emptyset,$
we set $S\coloneqq P[N(M_{0})^{\bot}]$ and $T\coloneqq(1-P)[N(M_{0})^{\bot}],$
where $P$ denotes the Dunford projection on $\sigma(A)\cap\mathbb{C}_{\Re<0}.$
The assertion then follows by \prettyref{lem:decay_L_2}. If on the
other hand, $\tilde{\mathcal{M}}$ admits an exponential dichotomy
with invariant subspaces $S,T\subseteq N(M_{0})^{\bot},$ we obtain
\[
\sigma(A)=\sigma(\iota_{S}^{\ast}A\iota_{S})\cup\sigma(\iota_{T}^{\ast}A\iota_{T})\subseteq\mathbb{C}_{\Re<0}\cup\mathbb{C}_{\Re>0}
\]
again by \prettyref{lem:decay_L_2}. Since $\sigma(\mathcal{M})=\sigma(A)$
by \prettyref{prop:regPen}, the assertion follows.
\end{proof}
From \prettyref{sec:spectra}, \prettyref{thm:strongDicho} and \prettyref{thm:mild_exp_dich}
we derive the following statement.
\begin{cor}
$\mathcal{M}$ admits a strong exponential dichotomy if and only if
it admits a mild exponential dichotomy. Moreover, $\mathcal{M}$ is
strongly exponentially stable if and only if it is mildly exponentially
stable.
\end{cor}

\bibliographystyle{abbrv}

\end{document}